\documentclass[12pt,letterpaper,reqno]{article}
\usepackage{graphicx}
\usepackage{amsmath}
\usepackage{amsfonts}
\usepackage{amsthm}
\usepackage{amssymb}
\usepackage[T1]{fontenc}

\theoremstyle{plain}
\newtheorem{theorem}{Theorem}[section]

\newtheorem{condition}[theorem]{Condition}

\newtheorem{corollary}[theorem]{Corollary}

\newtheorem{definition}[theorem]{Definition}

\newtheorem{lemma}[theorem]{Lemma}

\newtheorem{proposition}[theorem]{Proposition}
\newtheorem{remark}[theorem]{Remark}

\usepackage{ulem,soul,xcolor}

\newcommand{\mr}{\mathbb{R}}
\newcommand{\mc}{\mathcal}
\newcommand{\tl}{\tilde}
\newcommand{\lan}{\langle}
\newcommand{\ran}{\rangle}
\newcommand{\diverg}{\mathrm{div}}
\newcommand{\tr}{\mathrm{tr}}
\newcommand{\mt}{\mathbb{T}}
\newcommand{\eps}{\varepsilon}

\numberwithin{equation}{section}

\title{Existence and uniqueness for stochastic 2D Euler flows with bounded vorticity}
\author{Zdzis\l aw Brze\'zniak, Franco Flandoli and  Mario Maurelli}
\date{}

\begin{document}

\maketitle

\begin{abstract}
The strong existence and the pathwise uniqueness of solutions with $L^{\infty}$-vorticity of the 2D stochastic Euler equations are proved. The noise is
multiplicative and it involves the first derivatives. A Lagrangian approach is implemented, where a stochastic flow solving a nonlinear flow equation is constructed. The stability under regularizations is also proved.

\end{abstract}

\section{Introduction}

The aim of this paper is to prove the strong existence and the pathwise uniqueness of $L^\infty$ solutions to the stochastic 2D Euler equation in vorticity form
\begin{equation}
d\xi+u^{\xi}\cdot\nabla\xi\,dt+\sum_{k=1}^{\infty}\sigma_{k}\cdot\nabla\xi\circ dW^{k}=0,\qquad\xi|_{t=0}=\xi_{0},
\label{Euler Strat}
\end{equation}
where the initial vorticity $\xi_{0}$ also belongs to the $L^{\infty}$ space. The equation above is subject to the periodic boundary conditions and thus can be reformulated as a problem on a 2-dimensional torus $\mathbb{T}^2=\big(\mathbb{R}/\mathbb{Z}\big)^2$, see for instance \cite[chapter 2]{Temam_1983}. In other words the space variable is assumed to be an element of $\mathbb{T}^2$  and all fields are assumed to be $1$-periodic (or simply defined on $\mathbb{T}^2$). The noise coefficients $\sigma_{k}$'s are bounded, regular enough, divergence-free vector fields, $(W^{k})^\infty_{k=1}$ is a family of independent Brownian motions and the velocity field $u^{\xi}$ is defined as
\begin{equation*}
u^{\xi}_t(x)=K*\xi_t(x)=\int_{\mt^2}K(x-y)\, \xi_t(y)\,dy,\ \ x\in\mathbb{T}^2,
\end{equation*}
where $K=\nabla^\perp G=(-\partial_2G,\partial_1G)$ and $G$ is the Green function of the Laplacian on the torus $\mathbb{T}^2$ with mean $0$, i.e.,
\begin{equation*}
u^\xi=-\nabla^\perp(-\Delta)^{-1}\xi.
\end{equation*}
We will also prove the stability of the solutions under regularization of the kernel $K$.

The Stratonovich form is the natural one for several reasons, including physical intuition related to the Wong-Zakai principle and the fact that an
It\^{o} term of the form $\sum_{k=1}^{\infty}\sigma_{k}\cdot\nabla\xi dW^{k}$ would require a compensating second order operator to hope for a well defined system, see \cite{Roz}. Besides, the Stratonovich form preserves the $L^2$ norm of the solution and is the right one to deal with manifold-valued SPDEs, see \cite{brzgolond}. However, for the opportunity of mathematical analysis, we will formally rewrite the equation in the It\^{o} form
\begin{eqnarray}
d\xi +u^\xi\cdot\nabla\xi\,dt &+&\sum^\infty_{k=1}\sigma_k\cdot\nabla\xi dW^k -\frac12\sum^\infty_{k=1}(\sigma_k\cdot\nabla)\sigma_k\cdot\nabla\xi \,dt\nonumber\\& =& \frac12\sum^\infty_{k=1}\tr[\sigma_k\sigma_k^\ast D^2\xi]\,dt,
\label{Itoform}
\end{eqnarray}
and we will give a rigorous interpretation of the latter one (under some simplified assumptions). Nonetheless it is useful to think sometimes heuristically in
form of the Stratonovich expression and it would be misleading to believe that the equation has a parabolic character due to the term $\tr\left(aD^{2}\xi\right)$ in the It\^{o} formulation.

The noise in equation (1.1) has a very special form, compared to general
abstract models of Stochastic Partial Differential Equations (SPDEs). Our aim
is not an abstract generality. We have chosen this noise for two
reasons. Firstly,  because it occupies a relevant position in the Mathematical Physics
literature on fluids and secondly because it is of transport type, hence allowing us to
use special tools related to the transport equations (flows, $L^{\infty}$-bounds).
The applied and theoretical literature on SPDEs driven by this type of noise is rich,
see for instance \cite{BaxHar}, \cite{BeGaKu}, \cite{CeVi}, \cite{FaGaVe}, \cite{Gaw}, \cite{KuMu},
\cite{LeRa}, \cite{MikRoz}, \cite{Kun2}, in particular for its relation with turbulent
transport of passive scalars and the so called Kraichnan model (\cite{Kra1}, \cite{Kra2}), one of the
most remarkable theories  where stochastic models have been applied with
success to explain phenomena in fluid mechanics. The transport structure of
the nonlinear deterministic part of the equation (the vorticity in 2D\ is only
transported) and of the stochastic part (Stratonovich choice is important
here), allow one to use stochastic the flows and to control the $L^{\infty}$-norm
of solutions (the vorticity) by the $L^{\infty}$-norm of initial conditions.
This control is $\omega$-wise, uniform also in $\omega$ in $\Omega$. Thus, having assumed
that initial vorticity is bounded, the solution is uniformly bounded in all
parameters (also $\omega$), opposite to several other stochastic cases, like the additive noise. This property is an important tool of our approach and it cannot be readily extended to other stochastic perturbations of the Euler
equations.

What concerns the theory of the deterministic Euler equations, the uniqueness for $L^{\infty}$-vorticity in the deterministic case is the celebrated result of Wolibner \cite{Wolibner} and Yudovich (\cite{Yud1}, \cite{Yud}). In addition to an excellent recent monograph \cite{Majda+Bertozzi} where some additional information about the trajectory method can be found, one should also mention more recent publications as for instance a recent review paper \cite{Chemin} by Chemin and a new approach to the old non-uniqueness results of Schaeffer and Shnirelman by  De Lellis and  Sz\`ekelyhidi in \cite{DeL+S}.

The literature on the stochastic Euler equations counts a number of works, including
\cite{Bes1}, \cite{Bes2}, \cite{Bes3}, \cite{BesFla}, \cite{BrzPes}, \cite{CapCut}, \cite{FGPEuler}, \cite{GH+V}, \cite{Kim1}, \cite{Kim2}, \cite{MikVal1}, \cite{MikVal2}, \cite{Yokoyama}. The
differences are in the structure of the noise, the results and topologies
involved and sometimes the domain and boundary conditions. A full discussion
is not possible so we limit ourselves to few remarks. Some of the works deal
with additive noise, some others with more general, namely multiplicative,
noise but not of the form treated here which involves the derivatives of the
solution, and one paper with noise with derivatives of the solution. When the
noise is additive, the theory is more complete, also because the equation can
be studied pathwise. First results were given in \cite{BesFla}, where the existence is
proved when the initial data belongs to the space $V$ and the solution is an $H$-valued
continuous and $V$-valued square integrable process (where $V$ is the space of
divergence free vector fields with finite enstrophy and $H$ is the space of
divergence free square integrable vector fields, periodic in an appropriate sense. However this solution, constructed pathwise on a given
probability space, is not known to be progressively measurable. Moreover, if
the vorticity of the initial data is bounded and the external forces
(deterministic and random) satisfy certain assumptions, the solution is proved
to be unique. These results in the additive noise case have been improved and
generalized in the interesting paper \cite{Kim1}, based on different techniques with
respect to \cite{BesFla}, which relaxes various regularity and boundary conditions on
the noise for the result of existence and uniqueness of solutions with bounded
vorticity and proves very careful measurability properties in the case of
solutions in $V$, those which are not necessarily unique. Let us also mention that in the additive noise case the more recent paper \cite{GlaSveVic} gives delicate
$L^{\infty}$-vorticity estimates on invariant measures for the stochastic
Navier-Stokes equations with and their inviscid limit. Multiplicative noise,
depending on the velocity field $u$ (and not on the gradient) has been
initially treated in the paper \cite{CapCut} by nonstandard analysis tools. That paper
is devoted to the stochastic Euler equations on a 2-dimensional torus and the
authors prove the existence of a solution on the Loeb space and the existence of a
corresponding notion of statistical solution and it does not deal with
the uniqueness. Then, in the paper \cite{BrzPes} the authors prove the existence (but again not
the uniqueness) of a solution to a problem with multiplicative noise as in \cite{CapCut}
and possibly unbounded domains, but the state space is the space the space $H$
intersected with the Sobolev space $H^{1,p}$ for $p>2$. In this way they are
able to prove the existence of solutions which are H\"{o}lder continuous with
respect to the space variables. The question of uniqueness in the case of
multiplicative noise, which was left open by these and other works (like \cite{Bes1}),
has been recently investigated in the paper \cite{GH+V}, however only for
one-dimensional Brownian motion (so that the Doss-Sussmann transformation can
be used). Finally, let us mention the recent paper \cite{Yokoyama} by Yokoyama, which is
the closest to our model (the stochastic Euler equations in Stratonovitch form
with the noise coefficients depending linearly on the gradient of the
solution), where the author proves the existence of a martingale solution with
the state space $V$; the paper does not deal with the uniqueness.

We solve here the problem in the space $L^{\infty}$ following the Lagrangian approach of \cite{MarPul}. It is based, in the stochastic case, on the investigation of the stochastic flow equation
\begin{eqnarray*}
\Phi_{t}(x)&=&x+\int_{0}^{t}\int_{\mathbb{T}^{2}}K(\Phi_{s}(x)-\Phi_{s}(y))\, \xi_{0}(y)dy\\
&+&\sum_{k}\int_{0}^{t}\sigma_{k}(\Phi_{s}(x))dW_{s}^{k},\ \ t\in[0,T],\ x\in\mathbb{T}^2
\end{eqnarray*}
which is a problem of interest in itself, even when the kernel $K$ is smooth. This equation is not trivial because of the global dependence of $\Phi_t (x)$ on $\left(  \Phi_{s}(y)\right)  _{y\in\mathbb{T}^{2}}$ and the difficulty to develop stochastic calculus (for instance a fixed point argument) in the space of (measure preserving, continuous) maps $\psi:\mathbb{T}^{2}\rightarrow\mathbb{T}^{2}$. The approach inspired by \cite{MarPul} allows us to study this equation and apply the result to the existence and the uniqueness of equation \eqref{Euler Strat} in $L^{\infty}$.

%
%

\subsection*{Acknowledgement}

We thank an anonymous referee for several remarks and suggestions, which help to clarify some points and improve the exposition. The first named  author would like to thank the Dipartimento di Matematica, Universit\`a di Pisa, for the kind hospitality.

\section{The main results}

Before stating the results, we list the hypotheses with some preliminary remarks.

%

\begin{condition}
In the paper, we will always assume that $\xi_0$, the initial vorticity, belongs to the space $L^\infty(\mt^2)$.
\end{condition}

\begin{condition}\label{condW}
The family of processes $W=(W^k)^\infty_{k=1}$ is a cylindrical Brownian motion (i.e.\ $W^k$'s are independent Brownian motions), defined on a probability space $(\Omega,\mc{A},P)$, with respect to the filtration $\mathbb{F}=(\mc{F}_t)_{t\in[0,T]}$.
\end{condition}

\begin{condition}\label{condsigma}
The vector fields $\sigma_k$'s are divergence-free and belong to $C^{0,1}(\mt^2)$ (Lipschitz periodic functions, hence a.e.\ differentiable); moreover the family $(\sigma_k)^\infty_{k=1}$ is in $W^{1,\infty}(\ell^2)$, that is
\begin{equation*}
L_\sigma^2:=\sup_{x\in\mt^2}\sum^\infty_{k=1}|\sigma_k(x)|^2+\left\|\sum^\infty_{k=1}|D\sigma_k|^2\right\|_{L^\infty}<+\infty.
\end{equation*}
We call $a(x):=\sum^\infty_{k=1}\sigma_k(x)\sigma_k(x)^\ast $ ($A^\ast $ denotes the transpose matrix of $A$). We assume also that $a\equiv cI_2$, where $c$ is a non-negative constant (possibly equal to $0$) and $I_2$ is the constant identity matrix.
\end{condition}

\begin{remark}\label{simplicity}
If $a(x)\equiv cI_2$, for all $x$ in $\mathbb{T}^2$, the It\^o formulation \eqref{Itoform} of the stochastic Euler Equations simplifies to
\begin{equation}
d\xi +u^\xi\cdot\nabla\xi\,dt +\sum^\infty_{k=1}\sigma_k\cdot\nabla\xi dW^k = \frac12 c\Delta\xi\,dt.\label{simplItoform}
\end{equation}
Indeed, since the the vector fields $\sigma_k$'s are divergence-free, the first order It\^o correction term, namely $\frac12\sum_k(\sigma_k\cdot\nabla)\sigma_k\cdot\nabla\xi\,dt$, disappears:
\begin{equation*}
\sum_k\sum_i\sigma_{k,i}(x)\partial_i\sigma_{k,j}(x)=\sum_i\partial_i(\sum_k\sigma_{k,i}(x)\sigma_{k,j}(x))=\sum_i\partial_ia_{ij}(x)=0.
\end{equation*}
\end{remark}

\begin{remark}
Condition $a(x)\equiv cI_2$, $x\in\mathbb{T}^2$, can be avoided at the price of requiring more regularity on the functions $\sigma_k$'s and of a few additional computations, which would obscure the main arguments. Indeed, the fact that $a$ is constant implies the absence of the first order It\^o correction term, which contains the derivatives of $\sigma$, and that the operator $\frac12\tr[aD^2]=\frac12c\Delta$ commutes with the convolution with a given function; this will avoid the use of a second order commutator lemma (not difficult but boring and requiring maybe more regularity on the $\sigma_k$'s).
\end{remark}

\begin{remark}
Let us briefly discuss examples of noise covered by the class above. The
trivial example of a noise term of the form $\nabla\xi_{t}\cdot dW_{t}$ where
$W$ is a 2-dimensional Brownian motion is covered by taking $\sigma_{k}%
=e_{k}$ for $k=1,2$ (where $\left(  e_{1},e_{2}\right)  $ is the canonical
basis of $\mathbb{R}^{2}$) and $\sigma_{k}=0$ for $k\geq3$. In this case it
should be noticed that the stochastic Euler equations can be reduced to the
classical deterministic ones by the simple transformation $\widetilde{\xi
}\left(  t,x\right)  =\xi\left(  t,x+W_{t}\right)  $. More than this one, we
are mainly motivated by the examples described in the Mathematical Physics
literature quoted in the Introduction, where the noise term has heuristically
the form $\nabla\xi\left(  t,x\right)  \cdot\partial_{t}W\left(  t,x\right)
$, for a space-dependent random field $W\left(  t,x\right)  $, Brownian in
time, with a given incremental covariance function $Q\left(  x,y\right)
=E\left[  W\left(  1,x\right)  \otimes W\left(  1,y\right)  \right]  $, sometimes
prescribed through its Fourier spectrum like $Q\left(  x-y\right)
=\int_{\mathbb{Z}^{2}}e^{ik\cdot x}f\left(  \left\vert k\right\vert \right)
dk$, for suitable functions $f$. A rigorous and simple way to deal with such
space-time noise (correlated in space) is the one adopted above, namely to
prescribe a sequence of independent real valued Brownian motions $W_{t}^{k}$
and a sequence of vector fields $\sigma_{k}\left(  x\right)  $. The
space-dependent noise $W\left(  t,x\right)  $ is then given by $W\left(
t,x\right)  =\sum_{k=1}^{\infty}\sigma_{k}\left(  x\right)  W_{t}^{k}$ and the
function $Q\left(  x,y\right)  =\sum_{k=1}^{\infty}\sigma_{k}\left(  x\right)
\otimes\sigma_{k}\left(  y\right)  $ is its incremental covariance. If one starts with a prescribed covariance function $Q(x,y)$ (with suitable properties), the $\sigma_k$'s are an orthonormal basis of a certain Hilbert space, thus their form is not explicitly given (though their existence is guaranteed, see \cite{BaxHar} and \cite{LeRa}).

However, we should notice that in comparison with the literature on the Kraichnan model of turbulent advection (related to the original Kraichnan's papers \cite{Kra1}, \cite{Kra2}), we impose regularity properties on the vector fields $\sigma_{k}$
which forbid us from considering certain singular examples treated there. The
covariance function $Q\left(  x,y\right)  $, corresponding to our case, is always relatively regular,
while it scales with fractional powers of $\left\vert x-y\right\vert $ in
Kraichnan model, see for example \cite{Gaw}, \cite{LeRa}.
\end{remark}


\begin{definition}
Let $\xi$ be an element of $L^\infty([0,T]\times\mt^2\times \Omega)$. We say that $\xi$ is $\mathbb{F}$-weakly progressively measurable if, for every $f$ in $L^1(\mt^2)$, the process $t\rightarrow\lan\xi_t,f\ran=\int_{\mt^2}f\xi_tdx$ is $\mathbb{F}$-progressively measurable.
\end{definition}

Given an element $w$ in $L^\infty(\mt^2)$, we will write
\begin{equation*}
u=u^w=K*w.
\end{equation*}
If $w$ is also time-dependent, we will write $u^w_t=u^{w_t}$. It is well known, see Corollary \ref{logLipu}, that $|u^w(x)-u^w(y)|\le L_K\|w\|_{L^\infty}|x-y|(1-\log|x-y|)$ for some constant $L_K$ if $|x-y|\le1$.

Now we give a precise definition of a solution. We use the It\^o formulation, having in mind Remark \ref{simplicity}. In what follows, $\lan f,g\ran:=\int_{\mt^2}fgdx$ denotes the scalar product in $L^2(\mt^2)$.

\begin{definition}
Let $\xi_0$ be in $L^\infty(\mt^2)$. A distributional $L^\infty$ solution to the stochastic Euler vorticity equation \eqref{simplItoform} is an $\mathbb{F}$-weakly progressively measurable element $\xi$ in $L^\infty([0,T]\times\mt^2\times \Omega)$, such that, for every $\varphi$ in $C^\infty(\mt^2)$, it holds $P$-a.s.
\begin{eqnarray}\nonumber
\lan\xi_t,\varphi\ran&=& \lan\xi_0,\varphi\ran+ \int^t_0\lan\xi_r,u^\xi_r\cdot\nabla\varphi\ran dr+ \sum_k\int^t_0\lan\xi_r,\sigma_k\cdot\nabla\varphi\ran dW_r\\
&+& \frac12\int^t_0\lan\xi_r,\tr[aD^2\varphi]\ran dr\ \ \forall t \in[0,T].  \label{distribvort}
\end{eqnarray}
\end{definition}

It is implicit in the definition that the process $\lan\xi_t,\varphi\ran$ has continuous trajectories.

\begin{remark}
If a process $\xi \in L^\infty([0,T]\times\mt^2\times \Omega)$ is weakly progressive measurable then so is the process $u^\xi\xi$. Indeed it implies that, for every $h$ in $L^1(\mt^2\times\mt^2)$, the process
\begin{equation}
t\mapsto \int_{\mt^2}\int_{\mt^2}\xi(t,x)\, \xi(t,y)h(x,y)dxdy
\end{equation}
is progressive measurable (this can be verified first for $h$ of the form $h(x,y)=f(x)g(y)$, then approximating every $h$ with sums of such separable functions). Now, for a test function $\varphi$, it is enough to write $\int u^\xi\xi\cdot\varphi dx$ as
\begin{equation*}
\int_{\mt^2}\int_{\mt^2}K(x-y)\, \xi_t(y)\, \xi_t(x)\cdot\varphi(x)dxdy
\end{equation*}
and take $h(x,y)=K(x-y)\varphi(x)$.
\end{remark}

The main result about the stochastic Euler vorticity equation is as follows.

\begin{theorem}\label{mainvort}
Given $\xi_0$ in $L^\infty(\mt^2)$ and the cylindrical Brownian motion $W$ (with the associated filtration), under Conditions \ref{condW} and \ref{condsigma} on the coefficients of the noise, the stochastic Euler vorticity equation \eqref{simplItoform} admits a unique $L^\infty$ distributional solution.
\end{theorem}

\begin{remark}
Notice that the filtration is given a-priori. Thus both the existence and the uniqueness are in the strong sense: there exists a solution $\xi$ adapted to the (completed) Brownian filtration (the smallest possible filtration) and any solution, defined on a possibly larger filtered space, must coincide with $\xi$. The same kind of existence and uniqueness will hold for every equation we will meet.
\end{remark}

Theorem \ref{mainvort} will be proved by solving the associated non-local SDE:
\begin{eqnarray}\nonumber
\Phi_t(x)&=& x +\int^t_0\int_{\mt^2}K(\Phi_r(x)-\Phi_r(y))\, \xi_0(y)\,dy\,dr \\
&+&\sum_k\int^t_0\sigma_k(\Phi_r(x))dW^k_r.\label{stocEulerflow}
\end{eqnarray}
Notice that here the drift, namely
\begin{equation}
u^\Phi(t,x)=\int_{\mt^2}K(x-\Phi_t(y))\, \xi_0(y)dy,\label{expr_u}
\end{equation}
depends on the whole flow.



\begin{definition}
\begin{itemize}
\item A stochastic continuous flow is a measurable map $\Phi:[0,T]\times\mt^2\times\Omega\rightarrow\mt^2$ such that, for a.e.\ $\omega$ in $\Omega$, $\Phi(\omega): [0,T]\times\mt^2  \to \mt^2$ is continuous and, for every $x  \in \mt^2$, the process $\Phi(x) :[0,T]\times\Omega\rightarrow\mt^2$ is progressively measurable.
\item A stochastic continuous flow $\Phi$ is said to be measure-preserving iff\footnote{To avoid any ambiguity here and other similar situations we assume that there there exists a measurable set $\tilde \Omega$ of full $\mathbb{P}$-measure, such that for all $\omega \in \tilde \Omega$ and every $t\in [0,T]$, the map $\Phi_t(\omega): \mt^2 \to \mt^2$ preserves the Lebesgue measure on $\mt^2$.}, for a.e.\ $\omega$,$\Phi_t(\omega): \mt^2 \to \mt^2$ preserves the Lebesgue measure on $\mt^2$ for every $t$.
\end{itemize}
\end{definition}


\begin{definition}
We say that a stochastic continuous flow $\Phi$ is solution to the SDE \eqref{stocEulerflow} if, for every $x$, the process $X:=\Phi(x)$ solves the SDE
\begin{equation}
dX=u^\Phi(X)\,dt+\sum_k\sigma_k(X)dW^k
\end{equation}
with initial condition $X_0=x$.
\end{definition}



\begin{theorem}\label{mainflow}
Given $\xi_0$ in $L^\infty(\mt^2)$ and the cylindrical Brownian motion $W$ (with the associated filtration), there exists a unique measure-preserving stochastic flow solution to equation \eqref{stocEulerflow}. This solution is a continuous flow $\Phi$ of class $C^\alpha$ in space and $C^\beta$ in time, for some $\alpha>0$ and for every $\beta<1/2$.
\end{theorem}

\begin{remark}
Actually the uniqueness holds in a larger class of flows, namely the class $SM$ defined at the beginning of Section \ref{SEf}, as it can be seen from the proof of Theorem \ref{mainflow}.
\end{remark}


\subsection{The strategy}

There are two ways to prove our results. We will develop mainly the one which requires the weakest regularity assumptions on the $\sigma_k$'s. This strategy will be as follows.

First we will prove that, for a log-Lipschitz random vector field $u$, the SDE
\begin{equation*}
dX_t=u(X_t)\,dt+\sum_k\sigma_k(X_t)dW^k_t
\end{equation*}
admits a unique solution, given by a stochastic measure-preserving continuous flow (Lemma \ref{stocauxlemma}). This includes the case of a ``linear'' version of \eqref{stocEulerflow}, where the drift is replaced by $u^\psi$ for some fixed stochastic flow $\psi$, see also next paragraph for notation. Then, using an iteration scheme, we will build a unique solution to \eqref{stocEulerflow}, reaching the assertion of Theorem \ref{mainflow}.

In the subsequent section, we will use Theorem \ref{mainflow} to prove Theorem \ref{mainvort}.

In the last section, we will show the second method: a ``trick'' allows us to reduce the stochastic case to a modified deterministic case. This seems to be more rapid but requires the $\sigma_k$'s to be at least $C^2$ (at least if one wants to use classical results), while the first method requires only a Lipschitz-type (precisely $W^{1,\infty}(\ell^2)$) hypothesis on the diffusion coefficients. That is why we will not develop this second method in all the details.

\subsection{Log-Lipschitz property of $K$ and other useful facts}


First we state the fundamental log-Lipschitz property for $K$ and the drift $u^\xi$. The key inequality \eqref{keyest} is stated in \cite[section 1.2]{MarPul}, and it follows from standard estimates of the Green function $G$, see e.g.\ \cite[section 4.2]{Aub}. For the completeness sake, we have recalled the proof in the appendix. For $r\ge0$, call
\begin{equation*}
\gamma(r)=r(1-\log r)1_{]0,1/e[}(r)+(r+(1/e))1_{[1/e,+\infty[}(r).
\end{equation*}

\begin{remark}
The following elementary properties of $\gamma$ will be of use: the function $\gamma$ is increasing, concave and for every $0<\eps<1/e$, we have
\begin{equation}
\gamma(r)\le -r\log\eps+\eps, \ \ \forall r\ge0.\label{gammaprop}
\end{equation}
\end{remark}

\begin{lemma}
The map $K$, introduced before, is an $L^p(\mt^2)$ divergence-free (in the distributional sense) vector field, for every $p<2$, and verifies for certain constants $L_{0,K}$, $L_K$:
\begin{eqnarray}
\|K\|_{L^1(\mt^2)}\le L_{0,K},&\nonumber\\
\int_{\mt^2}|K(x-y)-K(x'-y)|dy\le L_K\gamma(|x-x'|),& \ \ \forall x,x'\in\mt^2.\label{keyest}
\end{eqnarray}
\end{lemma}

The divergence-free property is a consequence of the fact that $K$ is orthogonal to a gradient of a scalar field.

\begin{corollary}\label{logLipu}
For every $w$ in $L^\infty(\mt^2)$, $u^w=K*w$ is divergence-free and satisfies
\begin{eqnarray}
\|u^w\|_{L^\infty}\le L_{0,K}\|w\|_{L^\infty},&\nonumber\\
|u^w(x)-u^w(x')|\le L_K\|w\|_{L^\infty}\gamma(|x-x'|),& \ \ \forall x,x'\in\mt^2.\label{logLip1}
\end{eqnarray}
\end{corollary}

We will use also the following elementary result. We recall that, for a finite signed measure $\mu$ on a space $E$ and a measurable map $F:E\rightarrow E'$, $\nu=F_\#\mu$ denotes the image measure of $\mu$ on $E'$, namely $\nu(A)=\mu(F^{-1}(A))$ for every measurable set $A$ in $E'$. Notice that $\nu$ is a finite signed measure and that $|\nu|\le F_\#|\mu|$ (since $|\nu|(A)\le F_\#|\mu|(A)$ for every $A$).

\begin{lemma}\label{abscont}
Let $F$ be a measure preserving map on $\mt^2$ and let $w$ be in $L^\infty(\mt^2)$. Let $\mu$ the (signed) measure on $\mt^2$ with density $w$ (with respect to the Lebesgue measure) and define $\nu=F_\#w$. Then $\nu$ has a density (denoted by $v$) with respect to Lebesgue measure and $\|v\|_{L^\infty}\le \|w\|_{L^\infty}$.
\end{lemma}

\begin{proof}
It is enough to prove the Lemma when $w$ is nonnegative. Since $F$ is measure-preserving, if $A$ is a set of zero Lebesgue measure, then $\mc{L}^2\{F\in A\}=\mc{L}^2(A)=0$, and so $\int_Ad\nu=\int_{\mt^2}1_A(F)wdx=0$. So $\nu$ admits a (nonnegative) density $v$. Now, taking $\eps>0$, $B=\{v>\|w\|_{L^\infty}+\eps\}$, we have
\begin{eqnarray*}
(\|w\|_{L^\infty}+\eps)\mc{L}^2(B)\le \int_{\mt^2}1_Bvdx=\int_{\mt^2}1_B(F)wdx\le\\
\le \|w\|_{L^\infty}\mc{L}^2\{F\in B\}=\|w\|_{L^\infty}\mc{L}^2(B),
\end{eqnarray*}
which implies that $\mc{L}^2(B)=0$. By arbitrariness of $\eps$, we get $\|v\|_{L^\infty}\le \|w\|_{L^\infty}$.
\end{proof}

Given Lemma \ref{abscont}, we will use often $v=F_\#w$ instead of $\nu=F_\#\mu$.

Finally some other notation. Let $\psi$ a measurable measure-preserving flow on $\mt^2$. With the notation in the previous section define $\xi^\psi_t=(\psi_t)_\#\xi_0$ (which is in $L^\infty([0,T]\times\mt^2)$ by Lemma \ref{abscont}) and $u^\psi=u^{\xi^\psi}$, which also reads
\begin{equation*}
u^\psi(t,x)=\int_{\mt^2}K(x-\psi_t(y))\, \xi_0(y)dy.
\end{equation*}
As already noticed, the SDE \eqref{stocEulerflow} reads as
\begin{equation*}
\Phi_t(x)= x +\int^t_0u^\Phi_r(\Phi_r(x))dr +\sum_k\int^t_0\sigma_k(\Phi_r(x))dW^k_r.
\end{equation*}

\begin{remark}
By the definition of $u^\psi$, Corollary \ref{logLipu} and Lemma \ref{abscont}, $u^\psi$ enjoys $\|u^\psi\|_{L^\infty}\le L_{0,K}\|\xi_0\|_{L^\infty}$ and the following log-Lipschitz property:
\begin{equation}
|u^\psi(x)-u^\psi(x')|\le L_K\|\xi_0\|_{L^\infty}\gamma(|x-x'|),\ \ \forall x,x'\in\mt^2.\label{logLip2}
\end{equation}
\end{remark}

Given $\lambda$ a positive constant and $z_0$ in $[0,1/e]$, we will also denote by $z^\lambda(t,z_0)$ (omitting the $\lambda$ when not necessary) the solution to the ODE
\begin{equation*}
z_t=z_0+\int^t_0\lambda\gamma_r(z_r)dr
\end{equation*}
This $z$ is unique and has the explicit formula
\begin{equation}
z(t,z_0)=z_0^{\exp[-\lambda t]}e^{1-\exp[-\lambda t]}1_{t<t_0}+(2e^{-1}\exp[\lambda(t-t_0)]-e^{-1})1_{t\ge t_0},\label{z}
\end{equation}
where $t_0=t_0(\lambda,z_0)=\frac{1}{\lambda}\log\frac{1-\log z_0}{2}$ is the time such that $z(t_0)=1/e$. Notice that, for $z_0$ in $[0,\exp[1-2e^{\lambda T}]]$, it holds $t_0\ge T$ and so, for $t$ in $[0,T]$,
\begin{equation}
z(t,z_0)\le ez_0^{\exp[-\lambda t]}.\label{est_z}
\end{equation}


\section{The deterministic case}

We first treat the deterministic case, in order to show the basic ideas. The scheme of the proof, strongly inspired by \cite{MarPul}, is a suitable rewriting of \cite{MarPul}, convenient for generalization to the stochastic case.

Euler flows in 2D (on the torus $\mt^2$) are described by the following non-local ODE:
\begin{equation}
\Phi_t(x)=x+\int^t_0\int_{\mt^2}K(\Phi_s(x)-\Phi_s(y))\, \xi_0(y)dy.\label{Eulerflow}
\end{equation}


%

Equation \eqref{Eulerflow} reads as $\dot{\Phi}=u^\Phi(\Phi)$ (with initial condition $\Phi_0=id$), notice that the drift is log-Lipschitz. That is why we consider the auxiliary equation (linear problem):
\begin{equation}
X^x_t=x+\int^t_0u(s,X^x_s)ds,\label{linearized}
\end{equation}
where $u$ is a fixed measurable vector field with the following property: for every $t$, $x,y$,
\begin{equation}
|u(t,x)-u(t,y)|\le L_u\gamma(|x-y|)\label{logLip}
\end{equation}
for some $L_u$ independent of $t,x,y$.

\begin{lemma}\label{auxlemma}
For every initial datum $x$, equation \eqref{linearized} has a unique solution. This solution is described by a (unique) flow $\psi$ of measure-preserving homeomorphisms of class $C^\alpha$ in space and Lipschitz in time, with $\alpha={\exp[-L_uT]}$.
\end{lemma}

\begin{proof}
The existence of a global solution (in $\mr^2$) to \eqref{linearized} follows from the Peano Theorem, since $u$ is continuous bounded. The uniqueness holds by the Osgood criterion (since $\int^\eps_0\gamma(r)^{-1}dr=+\infty)$) or even by the H\"older estimate below (simply take $x=y$).

The Lipschitz continuity in time follows by boundedness of $u$. As for the H\"older continuity, property \eqref{logLip} implies that, for every $x$ and $x'$,
\begin{equation*}
|\psi_t(x)-\psi_t(x')|\le |x-x'|+L_u\int^t_0\gamma(|\psi_s(x)-\psi_s(x')|)ds.
\end{equation*}
By a comparison result, $|\psi_t(x)-\psi_t(x')|\le z^{L_u}(t,|x-x'|)$ (recall that $z^\lambda$ is the unique solution to $z_t=z_0+\int^t_0\lambda\gamma(z_s)ds$). The bound \eqref{est_z} for $z$ gives the desired regularity. The invertibility and the continuity of the inverse map are due to the classical cocycle law, so that the inverse flow of $\psi_t$ is $\psi_{-t}$. The measure-preserving property follows by a simple approximation argument, see the proof of Lemma \ref{stocauxlemma} in the stochastic case.
\end{proof}


Now we use the Picard iteration scheme to prove the existence and the uniqueness of solutions to \eqref{Eulerflow}. Consider the set
\begin{eqnarray*}
M_T&=&\Big\{\psi:[0,T]\times\mt^2\rightarrow\mt^2|\psi\mbox{ measurable }, \sup_{[0,T]}\int_{\mt^2}|\psi_t(x)|\,dx<+\infty,\\
&& \psi_t\mbox{ measure-preserving for a.e. }t\Big\}.
\end{eqnarray*}
It is a complete metric space, endowed with the distance $dist(\psi^1,\psi^2)=\sup_{[0,T]}\int_{\mt^2}|\psi^1_t(x)-\psi^2_t(x)|\,dx$. For any $\psi$ in $M_T$, define $G(\psi)$ as the unique flow solution to \eqref{linearized} with $u=u^\psi$, i.e.
\begin{equation*}
\frac{d}{dt}{G(\psi)}=u^\psi(G(\psi))
\end{equation*}
(with initial condition $G(\psi)(0,x)=x$). Recall that
\begin{equation*}
u^\psi(t,x)=\int_{\mt^2}K(x-\psi_t(y))\, \xi_0(y)dy.
\end{equation*}
enjoys the log-Lipschitz property \eqref{logLip2}, so that by the previous Lemma $G$ takes values in $M_T$.


\begin{lemma}\label{keylemma}
For every $\eps>0$, for every two flows $\psi^1$, $\psi^2$ in $M_T$, we have:
\begin{eqnarray}
\lefteqn{\int_{\mt^2}|G(\psi^1)_t(x)-G(\psi^2)_t(x)|\,dx}\nonumber\\
&\le& L_K\|\xi_0\|_{L^\infty}\int^t_0\gamma\left(\int_{\mt^2}|\psi^1_s(x)-\psi^2_s(x)|\,dx\right)ds
\nonumber \\
 &+& L_K\|\xi_0\|_{L^\infty}\int^t_0\gamma\left(\int_{\mt^2}|G(\psi^1)_s(x)-G(\psi^2)_s(x)|\,dx\right)ds\label{est1}
\end{eqnarray}
and also
\begin{eqnarray}
\lefteqn{\int_{\mt^2}|G(\psi^1)_t-G(\psi^2)_t|\,dx}\nonumber\\
&\le& L_K\|\xi_0\|_{L^\infty}(-\log\eps)\int^t_0\int_{\mt^2}|G(\psi^1)_s-G(\psi^2)_s|\,dxds
\nonumber\\
&+& L_K\|\xi_0\|_{L^\infty}(-\log\eps)\int^t_0\int_{\mt^2}|\psi^1_s-\psi^2_s|\,dxds + 2L_K\|\xi_0\|_{L^\infty}t\eps.\label{est2}
\end{eqnarray}
\end{lemma}

\begin{proof}
We have
\begin{eqnarray*}
&&\hspace{-1.9truecm}\lefteqn{\int_{\mt^2}|G(\psi^1)_t(x)-G(\psi^2)_t(x)|\,dx  \leq\|\xi_0\|_{L^\infty}\int^t_0\int_{\mt^2}\int_{\mt^2}} \\
&&\hspace{-0.9truecm}|K(G(\psi^1)_s(x)-\psi^1_s(y))-K(G(\psi^2)_s(x)-\psi^2_s(y))|\,dxdyds.
\end{eqnarray*}
In order to use \eqref{keyest}, we add and subtract $K(G(\psi^1)_s(x)-\psi^2_s(y))$ to the integrand of the right-hand side. Thus we get
\begin{eqnarray*}
&&\hspace{-0.9truecm}\lefteqn{\int_{\mt^2}|G(\psi^1)_t(x)-G(\psi^2)_t(x)|\,dx \le \|\xi_0\|_{L^\infty}\int^t_0\int_{\mt^2}\int_{\mt^2}}  \\
&&\Big[ |K(G(\psi^1)_s(x)-\psi^1_s(y))-K(G(\psi^1)_s(x)-\psi^2_s(y))|\nonumber\\
& & +|K(G(\psi^1)_s(x)-\psi^2_s(y))-K(G(\psi^2)_s(x)-\psi^2_s(y))|\Big] dxdyds\nonumber\\
&\le& \|\xi_0\|_{L^\infty}\int^t_0\int_{\mt^2}\int_{\mt^2} \Big[ |K(x-\psi^1_s(y))-K(x-\psi^2_s(y))|\nonumber\\
& & K(G(\psi^1)_s(x)-y)-K(G(\psi^2)_s(x)-y)|\Big] \,dxdyds\nonumber\\
&\le& L_K\|\xi_0\|_{L^\infty}\int^t_0\int_{\mt^2}\gamma(|\psi^1_s(y)-\psi^2_s(y)|)dyds\\
 &+& L_K\|\xi_0\|_{L^\infty}\int^t_0\int_{\mt^2}\gamma(|G(\psi^1)_s(x)-G(\psi^2)_s(x)|)dxds,\nonumber
\end{eqnarray*}
where in the second passage we used the measure-preserving property. Finally, by the Jensen inequality applied to the concave function $\gamma$, we have
\begin{eqnarray*}
\lefteqn{\int_{\mt^2}|G(\psi^1)_t(x)-G(\psi^2)_t(x)|\,dx}\\
&\le& L_K\|\xi_0\|_{L^\infty}\int^t_0\gamma\left(\int_{\mt^2}|\psi^1_s(x)-\psi^2_s(x)|\,dx\right)ds\\ &+& L_K\|\xi_0\|_{L^\infty}\int^t_0\gamma\left(\int_{\mt^2}|G(\psi^1)_s(x)-G(\psi^2)_s(x)|\,dx\right)ds,\nonumber
\end{eqnarray*}
that is the first estimate \eqref{est1}. Now we apply property \eqref{gammaprop}:
\begin{eqnarray*}
\lefteqn{\int_{\mt^2}|G(\psi^1)_t(x)-G(\psi^2)_t(x)|\,dx}\\
&\le& L_K\|\xi_0\|_{L^\infty}(-\log\eps)\int^t_0\int_{\mt^2}|\psi^1_s(x)-\psi^2_s(x)|\,dxds + \nonumber\\
& & + L_K\|\xi_0\|_{L^\infty}(-\log\eps)\int^t_0\int_{\mt^2}\int_{\mt^2}|G(\psi^1)_s(x)-G(\psi^2)_s(x)|\,dxds +\nonumber\\
& & + 2L_K\|\xi_0\|_{L^\infty}t\eps,\nonumber
\end{eqnarray*}
i.e.\ the second estimate \eqref{est2}.
\end{proof}

The following continuity result is a consequence of the previous Lemma.

\begin{corollary}\label{cor-continuity}
The map $G:M_T \to M_T$ is continuous. In fact, it is locally H\"older continuous.
\end{corollary}

\begin{proof}
Let us denote $w_t=\int_{\mt^2}|G(\psi^1)_t(x)-G(\psi^2)_t(x)|\,dx$. Then the estimate \eqref{est1} in Lemma \ref{keylemma}, together with monotonicity of $\gamma$, gives
\begin{eqnarray*}
w_t &\le& L_K\|\xi_0\|_{L^\infty}T\gamma\left(\sup_{s\in[0,T]}\int_{\mt^2}|\psi^1_s(x)-\psi^2_s(x)|\,dx\right)
\\
 &+& L_K\|\xi_0\|_{L^\infty}\int^t_0\gamma(w_s)ds.
\end{eqnarray*}
Again by a comparison theorem (recall the definition of $z$ in \eqref{z}), we get that
\begin{eqnarray*}
&&\hspace{-10truecm}\lefteqn{\int_{\mt^2}|G(\psi^1)_t(x)-G(\psi^2)_t(x)|\,dx}\\
&&\hspace{-8truecm}\lefteqn{\leq  z^{L_K\|\xi_0\|_{L^\infty}}\Big(t,L_K\|\xi_0\|_{L^\infty}T\gamma\big(\sup_{s\in[0,T]}\int_{\mt^2}|\psi^1_s(x)-\psi^2_s(x)|\,dx\big)\Big)}.
\end{eqnarray*}
When $L_K\|\xi_0\|_{L^\infty}T\gamma(dist(\psi^1,\psi^2))\le \exp[1-2e^{L_K\|\xi_0\|_{L^\infty}T}]$ (a condition which is verified for $dist(\psi^1,\psi^2)$ small enough, for fixed $\|\xi_0\|$ and $T$), the estimate \eqref{est_z} gives:
\begin{eqnarray*}
&&\hspace{-10truecm}\lefteqn{\int_{\mt^2}|G(\psi^1)_t(x)-G(\psi^2)_t(x)|\,dx}\\
&&\hspace{-8truecm}\lefteqn{\le e\left(L_K\|\xi_0\|_{L^\infty}T\gamma\left(\sup_{s\in[0,T]}\int_{\mt^2}|\psi^1_s(x)-\psi^2_s(x)|\,dx\right)\right)^{\exp[-L_K\|\xi_0\|_{L^\infty}t]}.}
\end{eqnarray*}
From this and the continuity of $\gamma$, we see that $G$ is continuous on $M_T$. The H\"older continuity of $G$ follows from the fact that $\gamma$ is H\"older continuous.
\end{proof}

We are ready to prove:

\begin{theorem}\label{existuniq} There exists a unique solution in $M_T$ to equation \eqref{Eulerflow}, which is a flow $\Phi$ of measure-preserving homeomorphisms of class $C^\alpha$ in space and Lipschitz in time.
\end{theorem}

\begin{proof}
\textbf{First step.} First we prove the existence and the uniqueness on an interval $[0,T_1]$, for $T_1$ small enough. For the existence, we define the approximating sequence for the solution to problem \eqref{Eulerflow}. Choose $\psi^0_t=I$. For any $n$, put  $\psi^{n+1}=G(\psi^n)$ ($G$ being defined on $M_{T_1}$) and denote $\rho^n_t=\sup_{k\ge n}\int_{\mt^2}|\psi^{k+1}_t(x)-\psi^k_t(x)|\,dx$. [The reason for the supremum in $k\ge n$ is to have, in the formula \eqref{uneq-inductive} below, $\rho^n$ on the left hand side and $\rho^{n-1}$ on the right hand side: otherwise it seems difficult to have good estimates.] The estimate \eqref{est2} in Lemma \ref{keylemma} gives immediately that for all $n\in\mathbb{N}$
\begin{equation}\label{uneq-inductive}
\rho^n_t\le 2L_K\|\xi_0\|_{L^\infty}(-\log\eps)\int^t_0\rho^{n-1}_sds + 2L_K\|\xi_0\|_{L^\infty}t\eps.
\end{equation}
By Lemma \ref{lem-ineq} we infer that
\begin{equation}
\sup_{[0,T_1]}\rho^n_t\le \frac{(2eL_K\|\xi_0\|_{L^\infty}T_1)^n}{\sqrt{2\pi n}}\sup_{[0,T_1]}\rho^0_t+2L_K\|\xi_0\|_{L^\infty}T_1\exp[n(2L_K\|\xi_0\|_{L^\infty}T_1-1)]\label{iterest}
\end{equation}
and so, provided $\alpha:=2eL_K\|\xi_0\|_{L^\infty}T_1<1$, there exists a unique $\psi \in M_{T_1}$ such that the sequence $(\psi^n)_n$ converges in $M_{T_1}$ to $\psi$. By Corollary \ref{cor-continuity} it follows that $G(\psi)=\psi$.

The uniqueness follows by applying the previous iterative scheme to two solutions $\Phi^1$, $\Phi^2$. More precisely we take $\Phi^{i,0}=\Phi^i$ and $\Phi^{i,n+1}=G(\Phi^{i,n})$, $i=1,2$, and we define $\bar{\rho}^n_t=\sup_{k\ge n}\int_{\mt^2}|\Phi^{1,k}_t(x)-\Phi^{2,k}_t(x)|\,dx$. Then \eqref{uneq-inductive} and so \eqref{iterest} hold for the sequence $\bar{\rho}^n$ (in place of $\rho^n$). But, since $\Phi^1$, $\Phi^2$ are solutions and hence fixed points of $G$, $\Phi^{i,n}=\Phi^i$ and $\bar{\rho}^n_t=\bar{\rho}^0_t$ for every $n$, $i=1,2$ and so we get
\begin{equation*}
dist(\Phi^1,\Phi^2)=\sup_{[0,T_1]}\bar{\rho}^n_t=\le \alpha^n \sup_{[0,T_1]}\bar{\rho}^0_t+e^{-1}\alpha e^{-n(1-\alpha)} =\alpha^n dist(\Phi^1,\Phi^2)+\alpha e^{-n(1-\alpha)},
\end{equation*}
for any integer $n$. Since $\alpha<1$, taking $n$ large, we get $dist(\Phi^1,\Phi^2)=0$.

\textbf{Second step.} We prove the global existence and uniqueness. They follow essentially by iteration in time, but we prefer to make this argument explicit, since the non-locality of the drift could create some confusion. The main point is to notice that, for fixed $0<T'<T$, a flow $\Phi$ solves the non-local ODE \eqref{Eulerflow} on $[0,T]$ if and only if it solves the non-local ODE on $[0,T']$ and it satisfies, for $t$ in $[T',T]$,
\begin{equation}
\Phi_t(x)=\Phi_{T'}(x)+\int^t_{T'}\int_{\mt^2}K(\Phi_s(x)-\Phi_s(y))\, \xi_0(y)dy.\label{Eulerflowiter}
\end{equation}
Hence we will prove the global result by showing the existence and the uniqueness for equation \eqref{Eulerflowiter} on $[T_1,2T_1]$, and then iterating the idea. As before, we define the approximating sequence $(\psi^n)_n$ of maps on $[T_1,2T_1]\times\mt^2$ by imposing
\begin{equation}
\psi^n_t(x)=\Phi_{T_1}(x)+\int^t_{T_1}\int_{\mt^2}K(\psi^n_s(x)-\psi^{n-1}_s(y))\, \xi_0(y)dy.\label{Eulerflowiterdef}
\end{equation}
Here a small technical clarification is needed for the existence, the continuity and the measure-preserving property of $\psi^n$: they cannot be inferred directly from Lemma \ref{auxlemma}, since the initial datum is no more $x$ (we could repeat the argument starting from $\Phi_{T_1}(x)$: this can be done, but at the price of introducing a flow map $\Phi_{T_1,t}$ which we avoid for simplicity). So we prove $\psi^n$ exists continuous and is measure-preserving, by defining $\psi^n$ on the whole interval $[0,2T_1]$ as $\psi^0_t=\Phi_t1_{[0,T_1]}+\Phi_{T_1}1_{]T_1,2T_1]}$ and $\psi^n=G(\psi^{n-1})$, the map $G$ relative to the interval $[0,2T_1]$. In this way $\psi^n$ coincides with $\Phi$ on $[0,T_1]$ (in particular it satisfies the condition $\psi^n_{T_1}=\Phi_{T_1}$) and it verifies equation \eqref{Eulerflowiterdef} on $[T_1,2T_1]$. The definition of $\psi^n$ (with continuity and measure-preserving property) is now done.\\
Having the existence and the measure-preserving property, we can repeat the estimates in Lemma \ref{keylemma}, starting from $\Phi_{T_1}$, with no difference in the proof; in particular the estimates hold with the same constant and with final time $T$ which is replaced by $T-T_1$. In this way we get the existence on $T_1\le t\le T_1+T_1=2T_1$.

The uniqueness follows again applying the iterative scheme above to two solutions and concluding as in step 1.

\textbf{Step 3.} The regularity and homeomorphism properties hold by Lemma \ref{auxlemma}, since $\Phi=G(\Phi)$ is in the image of $G$ ($G$ now being defined on the whole $[0,T]$).
\end{proof}

\begin{remark} In case $\xi_0$ is more smooth, more regularity of $\Phi$ can be obtained, using the usual iterative scheme: if $\Phi$ has some regularity, then $u^\Phi$ has more regularity, which implies again an improvement of regularity of $\Phi$, and so on.
\end{remark}

\section{The stochastic case}




Now we prove the existence and the uniqueness of a stochastic continuous flow solving equation \eqref{stocEulerflow}. Notice that, differently from the classical (linear) case, the drift depends on the whole flow, so Kunita's theory (\cite{K}, \cite{Kun2}) is not (at least easily) applicable.

We try to mimic the previous reasoning in the deterministic case. The last part, the iterative procedure from the proof of Theorem \ref{existuniq}, works in this simple way. First we get a generalized Lemma \ref{auxlemma} (with It\^o formula to treat the modulus of the difference of two flows), then we repeat the scheme and obtain a measurable flow solution to \eqref{stocEulerflow}.

The main difficulty is in the first part, precisely in the generalization of Lemma \ref{auxlemma} to stochastic continuous flows (remember that we need a continuity property for $\omega$ fixed). In order to get rid of the first difficulty, we will apply Kolmogorov test, in the spirit of Kunita's results (see \cite{K}, \cite{Kun2}). For this we need some estimates on the linear equation.

%
%

\subsection{The linear stochastic equation}

Consider the following SDE (``linear'' problem):
\begin{equation}
dX_t=u_t(X_t)\,dt+\sum_k\sigma_k(X_t)dW^k_t,\label{stoclinearized}
\end{equation}
where $u$ is a random vector field with the following properties: for every $x$, $(t,\omega)\rightarrow u(t,x,\omega)$ is a progressively measurable process and, for every $t$, $x,y$, $\omega$,
\begin{eqnarray}
u(t,x,\omega)\le L_{0,u},\\
|u(t,x,\omega)-u(t,y,\omega)|\le L_u\gamma(|x-y|)\label{bddu}
\end{eqnarray}
for some $L_{0,u}$, $L_u$ independent of $t,x,y,\omega$. These properties imply that, if $X$ is a progressively measurable process with values in $\mt^2$, then $u(t,X_t)$ is progressively measurable too.

\begin{lemma}\label{contlemma}
Let $X$, $Y$ be two solutions of \eqref{stoclinearized} starting from $x$, $x'$ resp.. Then, for any $p\ge2$, there exists $\delta=\delta(T,L_u,L_\sigma,p)$ such that, if $|x-x'|<\delta$, it holds for some constant $C_{p,T}$, depending only on $p$ and $T$
\begin{equation}
E[|X_t-Y_{t'}|^p]\le e|x-x'|^{p\exp[-(2pL_u+L_\sigma)T]}+C_{p,T}(L_{0,u}^p+L_\sigma^p)|t-t'|^{p/2}.\label{contest}
\end{equation}
\end{lemma}


\begin{proof}
It is enough to prove the formula in the two particular cases $t=t'$ and $x=x'$. Fix $t=t'$. By the It\^o formula (applied to $f(x)=|x|^p$), calling $Z=X-Y$, we have
\begin{eqnarray*}
\lefteqn{d[|Z|^p]=p|Z|^{p-2}Z\cdot(u(X)-u(Y))\,dt}\nonumber\\
& +&\Big[\sum_kp|Z|^{p-2}|\sigma_k(X)-\sigma_k(Y)|^2\Big]\,dt
\nonumber\\
&+& \Big[\sum_kp(p-2)|Z|^{p-4}|Z\cdot(\sigma_k(X)-\sigma_k(Y))|^2\Big]\,dt+\nonumber\\
& +&\sum_kp|Z|^{p-2}Z\cdot(\sigma_k(X)-\sigma_k(Y))dW^k.
\end{eqnarray*}
We take the expectation and use the Lipschitz continuity of $\sigma_k$'s and the log-Lipschitz property of $u$:
\begin{eqnarray*}
\lefteqn{p|Z|^{p-1}|u(X)-u(Y)|}\\
&\le& pL_u|Z|^p(1-\log|Z|)1_{|Z|<1/e}+pL_u|Z|^{p-1}(|Z|+1/e)1_{|Z|\ge1/e}\\
&\le& pL_u|Z|^p(1-\log|Z|^p)1_{|Z|<1/e}+2pL_u|Z|^p+1/e1_{|Z|\ge1/e}pL_u\gamma(|Z|^p)\\
&\le& 2pL_u\gamma(|Z|^p).
\end{eqnarray*}
Then
\begin{equation*}
E[|Z|^p_t]\le |x-x'|^p+ 2pL_u\int^t_0E[\gamma(|Z|^p_s)]ds+ L_\sigma\int^t_0E[|Z|^p_s]ds,
\end{equation*}
from which, using Jensen inequality for the concave function $\gamma$ and the fact that $r\le \gamma(r)$, we obtain
\begin{eqnarray*}
\lefteqn{E[|Z|^p_t]\le |x-x'|^p+ 2pL_u\int^t_0\gamma(E[|Z|^p_s])ds+ L_\sigma\int^t_0E[|Z|^p_s]ds}\\
&\le& |x-x'|^p+ (2pL_u+L_\sigma)\int^t_0\gamma(E[|Z|^p_s])ds.
\end{eqnarray*}
By a comparison principle, $E[|Z|^p_t]\le z^{2pL_u+L_\sigma}(t,|x-x'|^p)$ (recall the definition of $z$ in \eqref{z}). When $|x-x'|$ is small enough (precisely, $<\delta$ for some $\delta(T,L_u,L_\sigma,p)$), we can apply the estimate \eqref{est_z} and we get the thesis for $t=t'$.

Now put $x=x'$, $t'<t$. By the boundedness of $u$ and $\sigma_k$'s, using the H\"older and the Burkholder inequalities, we get
\begin{eqnarray*}
E[|X_t-X_{t'}|^p]&\leq & 2^{p-1}E\Big[\big|\int^t_{t'}u(X_r)dr\big|^p
+\big|\sum_k\int^t_{t'}\sigma_k(X_r)dW^k_r\big|^p\Big]
\\
&\le& C_p(L_{0,u}^p+L_\sigma^{p,T})(|t-t'|^p+|t-t'|^{p/2}).
\end{eqnarray*}
The proof is complete.
\end{proof}

%

This will be enough to get the uniqueness and the continuity, but we still need the existence. For this, we will use a generalization of the previous lemma, exhibiting a Cauchy sequence of solutions of approximating equations. Let $\rho$ be a $C^\infty_c(\mr^2)$ function, define $\rho_\eps(x)=\eps^{-2}\rho(\eps^{-1}x)$; consider the standard mollification of $u$: $u^\eps(t,x,\omega)=u(t,\cdot,\omega)*\rho_\eps(x)$, for $x\in\mt^2$ (the convolution must be understood on the whole $\mr^d$, where $u$ is extended by periodicity). Notice that, since by \eqref{bddu} the field $u$ is continuous and bounded in $x$, uniformly with respect to $t$ and $\omega$, we get that $(u^\eps)_\eps$ converges to $u$ uniformly in $(t,x,\omega)$: that is, we can find a continuous function $\theta:[0,1]\rightarrow[0,+\infty[$, with $\theta(0)=0$, such that, for every $\eps>0$, $\delta>0$,
\begin{equation}
\sup_{[0,T]\times\mt^2\times\Omega}|u^\eps-u^\delta|\le \theta(|\eps-\delta|).\label{unifapprox}
\end{equation}
Moreover, Corollary \ref{cor-continuity} holds uniformly in $\eps$:
\begin{equation}
\sup_{\eps>0}|u^\eps(t,x)-u^\eps(t,x')|\le L_u\gamma(|x-x'|).\label{unifapprox2}
\end{equation}
Similarly, we define $\sigma_k^\eps(t,x):=\sigma_k(t,\cdot)*\rho_\eps(x)$; since the $\sigma_k$'s are Lipschitz-continuous (more precisely, by Condition \ref{condsigma}), we get (possibly for another $\theta$, with the same properties as above)
\begin{eqnarray}
\sup_{[0,T]\times\mt^2}\sum_k|\sigma_k^\eps-\sigma_k^\delta|^2\le \theta(|\eps-\delta|),\label{sigmaapprox}\\
\sup_{\eps>0}\sum_k|\sigma_k^\eps(t,x)-\sigma_k^\eps(t,x')|^2\le L_\sigma^2|x-x'|^2.\label{sigmaapprox2}
\end{eqnarray}

\begin{lemma}\label{existlemma}
For any $\eps>0$, let $\psi^\eps$ be the stochastic continuous flow solution to
\begin{equation}
dX_t^\eps=u^\eps_t(X_t^\eps)\,dt+\sum_k\sigma_k^\eps(X_t^\eps)dW^k_t.\label{approxstoclinearized}
\end{equation}
Then, for any $p\ge2$, for every $\eps$, $\delta$ close enough to $0$, for every $x$, $x'$ in $\mt^2$ with $|x-x'|$ small enough, it holds
\begin{equation*}
\sup_{[0,T]}E[|\psi^\eps_t(x)-\psi^\delta_t(x')|^p]\le C\left(|x-x'|^p+ C\theta(\eps-\delta)\right)^{\exp[-Ct]}
\end{equation*}
for some $C>0$ (independent of $\eps,\delta,x,x'$). In particular, $(\psi^\eps)_\eps$ is a Cauchy sequence in $C([0,T]\times\mt^2;L^p(\Omega))$.
\end{lemma}

For the sake of simplicity, we do not specify, in the result above and in the proof below, the constants involved (using the letter $C$ for all of them), since the estimates will not be used in the proof of the main result.


\begin{remark}
For every $\eps>0$, for every initial datum, equation \eqref{approxstoclinearized} has a unique solution, which can be represented by a stochastic continuous flow $\psi^\eps$ of $C^1$ maps. Indeed, by the boundedness of $u$, the $C^1$ norm of $u^\eps$ is uniformly bounded, and Kunita's theory applies. Notice that here we need Kunita's result with a stochastic drift, namely \cite{Kun2}, Theorem 4.6.5.
\end{remark}

\begin{remark}\label{approxmeaspres}
Again for $\eps>0$, since the stochastic integral is of Stratonovich type (which we have written in It\^o form), usual calculus rules give the standard equation for the Jacobian, which depends only on the divergence of the vector fields. Since $u^\eps$ and $\sigma_k^\eps$'s are divergence free, the Jacobian turns out to be constant and so the stochastic flow is measure-preserving.
\end{remark}

\begin{proof}
First we notice that, for $p\ge2$, $E[|\psi_t^\eps(x)|^p]$ is bounded by a constant independent of $\eps$, $t$ and $x$ (simply estimate $|u^\eps(X^\eps)|$ and $|\sigma_k(X^\eps)|$ with the $\sup$-norms of $u$ and $\sigma_k$ and use H\"older and Burkholder inequalities). Similarly, one sees that $\psi^\eps$ is in $C([0,T]\times\mt^2;L^p(\Omega))$ for every $\eps>0$. By It\^o formula (applied to $f(x)=|x|^p$), calling $Z=\psi^\eps_t(y)-\psi^\delta_t(x)$, we have
\begin{eqnarray*}
\lefteqn{d[|Z|^p]=p|Z|^{p-2}Z\cdot(u^\eps(\psi^\eps(x))-u^\delta(\psi^\delta(x')))\,dt}\nonumber\\
& +&\Big[\sum_kp|Z|^{p-2}|\sigma_k^\eps(\psi^\eps(x))-\sigma_k^\delta(\psi^\delta(x'))|^2 \nonumber\\
&+&\sum_kp(p-2)|Z|^{p-4}|Z\cdot(\sigma_k^\eps(\psi^\eps(x))-\sigma_k^\delta(\psi^\delta(x')))|^2\Big]\,dt\nonumber\\
& +&\sum_kp|Z|^{p-2}Z\cdot(\sigma_k^\eps(\psi^\eps(x))-\sigma_k^\delta(\psi^\delta(x')))dW^k.\nonumber
\end{eqnarray*}
The difficult term is $u^\eps(\psi^\eps(x))-u^\delta(\psi^\delta(x'))$. For this, by \eqref{unifapprox} and \eqref{unifapprox2}, we have
\begin{eqnarray*}
\lefteqn{|u^\eps(\psi^\eps(x))-u^\delta(\psi^\delta(x'))|}\\
&\le& |u^\eps(\psi^\eps(x))-u^\delta(\psi^\eps(x))|+|u^\delta(\psi^\eps(x))-u^\delta(\psi^\delta(x'))|\\
&\le& \theta(\eps-\delta)+C\gamma(|Z|).
\end{eqnarray*}
The terms with $\sigma_k^\eps$ are easier: by \eqref{sigmaapprox} and \eqref{sigmaapprox2}, we have
\begin{eqnarray*}
\lefteqn{\sum_k|\sigma^\eps(\psi^\eps(x))-\sigma_k^\delta(\psi^\delta(x'))|^2}\\
&\le& 2\sum_k\left[|\sigma_k^\eps(\psi^\eps(x))-\sigma_k^\delta(\psi^\eps(x))|^2+|\sigma_k^\delta(\psi^\eps(x))-\sigma_k^\delta(\psi^\delta(x'))|^2\right]\\
&\le& \theta(\eps-\delta)+C|Z|^2.
\end{eqnarray*}
So, proceeding as before, using concavity of $\gamma$ and uniform boundedness of $E[|Z|^{p-1}]$ and $E[|Z|^{p-2}]$, we get
\begin{equation*}
E[|Z|^p_t]\le |x-x'|^p+ C\theta(|\eps-\delta|)+ \int^t_0\gamma(E[|Z|^p_s])ds.
\end{equation*}
We conclude that, if $|x-x'|^p+ C\theta(|\eps-\delta|)$ is small enough (precisely, smaller that a constant depending on $T$, $C$ and $p$),
\begin{equation*}
\sup_{[0,T]}E[|Z|^p_t]\le C\Big(|x-x'|^p+ C\theta(|\eps-\delta|)\Big)^{\exp[-Ct]},
\end{equation*}
which implies that, if $x=x'$, the sequence $(\psi^\eps)_\eps$ is Cauchy in the space $C([0,T]\times\mt^2;L^p(\Omega))$.
\end{proof}

\begin{lemma}\label{stocauxlemma}
Equation \eqref{stoclinearized} has a unique solution, for every deterministic initial datum. This solution is described by a (unique) stochastic measure-preserving continuous flow $\psi$ of class $C^\alpha$ in space, for some $\alpha>0$, and $C^\beta$ in time, for every $\beta<1/2$.
\end{lemma}

\begin{proof}
By the previous Lemma, for every $x$, there exists the limit,  in $C([0,T];L^p(\Omega))$,  $X$ of the approximating processes $X^\eps=\psi^\eps(x)$'s. Then we can pass to the limit in equation \eqref{approxstoclinearized}, because the coefficients are continuous bounded. Hence we infer that the process $X$ is progressively measurable and solves \eqref{stoclinearized}. The uniqueness follows from Lemma \ref{contlemma}, with $x=y$. The H\"older continuity property is a consequence of the Kolmogorov criterion, applied again to \eqref{contest}. Indeed we get that $\psi$ is $\alpha$-H\"older continuous in space, for every $\alpha<e^{-CT}-2/p$, and $\beta$-H\"older continuous in time, for every $\beta<1/2-1/p$, so for every $\beta<1/2$.

As for the measure-preserving property, we will prove that, for every bounded measurable $F:\Omega\rightarrow\mr$, every bounded measurable $h:[0,T]\rightarrow\mr$ and every continuous bounded $g:\mt^2\rightarrow\mr$,
\begin{equation}
\int^T_0h(t)E\left[F\int_{\mt^2}g(\psi_t(x))dx\right]\,dt=\int^T_0h(t)E\left[F\int_{\mt^2}g(x)dx\right]\,dt.\label{measpres}
\end{equation}
This will prove that, for a.e.\ $(t,\omega)$, $\psi_t(\omega)$ is measure-preserving. By continuity in $(t,x)$ at $\omega$ fixed, this implies easily that, for a.e. $\omega$, $\psi_t(\omega)$ is measure-preserving for every $t$. Since the approximating flows $\psi^\eps$'s are measure-preserving (remember Remark \ref{approxmeaspres}), equality \eqref{measpres} holds for the $\psi^\eps$'s. By the convergence in $L^p$, we can find a subsequence $\psi^{\eps_n}$ such that $(\psi^{\eps_n})_n$ converges to $\psi$ for a.e.\ $(t,x,\omega)$. Passing to the limit along this subsequence (using dominated convergence theorem), we get \eqref{measpres} for $\psi$. The proof is complete.
\end{proof}


\begin{remark}
With a small effort, one could also show the injectivity of $\psi_t(\omega)$ for all $t$, for a.e.\ $\omega$ (essentially, one has to extend Lemma \ref{contlemma} to negative $p$ and use Kolmogorov criterion for $|\psi_t(x)-\psi_t(y)|^{-1}$). Surjectivity and continuity of the inverse map follow from the continuity and the measure-preserving property. The range of a measure-preserving continuous map is a compact set, whose complement (an open set) is Lebesgue-negligible. Thus this range must be the whole $\mt^2$. Thus the flow is a actually a flow of homeomorphisms.
\end{remark}

\begin{corollary}\label{stocu}
Let $\xi$ be an element of $L^\infty([0,T]\times\mt^2\times\Omega)$. Then equation \eqref{stoclinearized} with $u=u^\xi$ has a unique solution, for every deterministic initial datum, which enjoys the properties in Lemma \ref{stocauxlemma}.
\end{corollary}

\subsection{Stochastic Euler flows}\label{SEf}

The rest of the section goes on in analogy with the deterministic case. We define a space
\begin{eqnarray*}
SM_T&=&\Big\{\psi:[0,T]\times\mt^2\times\Omega\rightarrow\mt^2:\psi\mbox{ measurable w.r.t.\ }\mc{P}\times\mc(B)(\mt^2),\\ &&\sup_{[0,T]}\int_{\mt^2}E[|\psi_t(x)|]dx<+\infty,\psi_t\mbox{ meas.-pres. for a.e. }(t,\omega)\Big\}.
\end{eqnarray*}
Here $\mc{P}$ is the predictable $\sigma$-algebra associated with the filtration $(\mc{F}_t)_t$. It is a complete metric space, endowed with the distance $dist(\psi^1,\psi^2)=\sup_{[0,T]}\int_{\mt^2}E|\psi^1_t(x)-\psi^2_t(x)|\,dx$. For a given measure-preserving stochastic flow $\psi$ in $SM_T$, we call $G(\psi)$ the unique solution to the SDE \eqref{stoclinearized} with $u=u^\psi$. Recall again that
\begin{equation*}
u^\psi(t,x)=\int_{\mt^2}K(x-\psi_t(y))\, \xi_0(y)dy.
\end{equation*}
enjoys the log-Lipschitz property \eqref{logLip2} and it is also progressively measurable as required in the previous section, so that $G$ takes values in $SM_T$.

\begin{remark}
One may ask at this point why, in the definition on $SM_T$, we have the supremum in time outside the expectation and not inside (while Burkholder inequality allows supremum inside, in some cases). The reason is that the argument works with the supremum outside and putting the supremum inside could create additional difficulties. A posteriori, since the flow solution $\Phi$ to \eqref{stocEulerflow} is in the image of $G$, it is continuous and also H\"older continuous.
\end{remark}

\begin{lemma}\label{stockeylemma} For every $\eps>0$ (small enough), for every $\psi^1$, $\psi^2$ flows in $SM_T$, the following estimates hold:
\begin{eqnarray}
\lefteqn{\int_{\mt^2}E|G(\psi^1)_t(x)-G(\psi^2)_t(x)|\,dx}\nonumber\\
&\le& L_K\|\xi_0\|_{L^\infty}\int^t_0\gamma\left(\int_{\mt^2}E|\psi^1_s(x)-\psi^2_s(x)|\,dx\right)ds\nonumber\\
 &+& L_K\|\xi_0\|_{L^\infty}\int^t_0\gamma\left(\int_{\mt^2}E|G(\psi^1)_s(x)-G(\psi^2)_s(x)|\,dx\right)ds,\nonumber\\
&  +& 2L_\sigma^2\int^t_0\int_{\mt^2}E|G(\psi^1)_s(x)-G(\psi^2)_s(x)|\,dxds,\nonumber\\
\lefteqn{\int_{\mt^2}E|G(\psi^1)_t-G(\psi^2)_t|\,dx}\nonumber\\
&\le& (L_K\|\xi_0\|_{L^\infty}+2L_\sigma^2)(-\log\eps)\int^t_0\int_{\mt^2}E|G(\psi^1)_s-G(\psi^2)_s|\,dxds\nonumber\\
&+& L_K\|\xi_0\|_{L^\infty}(-\log\eps)\int^t_0\int_{\mt^2}E|\psi^1_s-\psi^2_s|\,dxds + 2L_K\|\xi_0\|_{L^\infty}t\eps\nonumber.
\end{eqnarray}
\end{lemma}

\begin{proof} We would like to apply It\^o formula to the modulus function and get an estimate for $|G(\psi^1)_t(x)-G(\psi^2)_t(x)|$. Since the modulus is not $C^2$, we use the approximate functions $f_\delta(x)=(|x|^2+\delta)^{1/2}$, for $\delta>0$. Calling $Z=G(\psi^1)_t(x)-G(\psi^2)_t(x)$, we have
\begin{eqnarray*}
&&\hspace{-1truecm}\lefteqn{d[f_\delta(Z)]= f_\delta(Z)^{-1}Z\cdot[u^{\psi^1}(G(\psi^1))-u^{\psi^2}(G(\psi^2))]\,dt}
 \\&+& \sum_kf_\delta(Z)^{-1}|\sigma_k(G(\psi^1))-\sigma_k(G(\psi^2))|^2dt  \nonumber\\
& +& \sum_kf_\delta(Z)^{-3}[(G(\psi^1)-G(\psi^2))\cdot(\sigma_k(G(\psi^1))-\sigma_k(G(\psi^2)))]^2dt  \nonumber\\
& +& \sum_kf_\delta(Z)^{-1}Z\cdot[\sigma_k(G(\psi^1))-\sigma_k(G(\psi^2))]dW.\nonumber
\end{eqnarray*}
Taking the expectation and using the Lipschitz property of $\sigma$, since $f_\delta(x)\ge|x|$, we get
\begin{eqnarray*}
E[|Z_t|]&\le& \int^t_0E[|u^{\psi^1}_s(G(\psi^1)_s(x))-u^{\psi^2}_s(G(\psi^2)_s(x))|]ds \\
&+& 2L_\sigma^2\int^t_0E[|Z_s|]ds.
\end{eqnarray*}
The rest of the proof follows the lines of Lemma \ref{keylemma}: we estimate $\int_{\mt^2}|u^{\psi^1}_s(G(\psi^1)_s(x))-u^{\psi^2}_s(G(\psi^2)_s(x))|\,dx$ and use Jensen inequality to pass $\gamma$ outside the integral in $x$ and outside the expectation. The second inequality is a consequence of the first one.
\end{proof}

\begin{proof}[Proof of Theorem \ref{mainflow}]
Similar to the proof of Theorem \ref{existuniq}, we only recall the main passages.

\textbf{First step}. We prove the existence and the uniqueness on an interval $[0,T_1]$, with $T_1$ small enough (but deterministic). The iteration scheme is completely similar to the one in the deterministic case: we consider $\psi^0_t(x)=x$, $\psi^{n+1}=G(\psi^n)$, $\rho^n_t=\sup_{k\ge n}\int_{\mt^2}E|\psi^{k+1}_t(x)-\psi^k_t(x)|\,dx$ and proceed as in the deterministic case, getting a limit flow $\Phi$ in $SM_{T_1}$, for $T_1$ such that $\alpha:=2e(L_K\|\xi_0\|_{L^\infty}+L_\sigma^2)T_1<1$ (notice that $T_1$ is independent of $\omega$, since all the estimates are in expectation). Such a flow solves \eqref{stocEulerflow}, because $G$ is continuous in $SM_T$: indeed, from Lemma \ref{stockeylemma} again by comparison with $z$
\begin{eqnarray*}
\lefteqn{\int_{\mt^2}E|G(\psi^1)_t(x)-G(\psi^2)_t(x)|\,dx} \\
& \leq& z^{L_K\|\xi_0\|_{L^\infty}+2L_\sigma^2}\Big(t,L_K\|\xi_0\|_{L^\infty}T\gamma\big(\sup_{s\in[0,T]}\int_{\mt^2}E|\psi^1_s(x)-\psi^2_s(x)|\,dx\big)\Big)
\end{eqnarray*}
and, if $dist(\psi^1,\psi^2)$ is small enough,
\begin{eqnarray*}
\lefteqn{\int_{\mt^2}E|G(\psi^1)_t(x)-G(\psi^2)_t(x)|\,dx} \\
&\leq& e\Big(L_K\|\xi_0\|_{L^\infty}T\gamma\big(\sup_{s\in[0,T]}\int_{\mt^2}E|\psi^1_s(x)-\psi^2_s(x)|\,dx\big)\Big)^{\exp[-(L_K\|\xi_0\|_{L^\infty}+2L_\sigma^2)t]}.\nonumber
\end{eqnarray*}
The uniqueness on $[0,T_1]$ is also proved in the same way of the deterministic case.

\textbf{Second step.} We prove the global existence and uniqueness. For this, as in the deterministic case, we solve the equation on $[T_1,2T_1]$
\begin{equation*}
\Phi_t(x)=\Phi_{T_1}(x)+\int^t_{T_1}\int_{\mt^2}K(\Phi_s(x)-\Phi_s(y))\, \xi_0(y)dy+\sum_k\int^t_{T_1}\sigma_k(\Phi_s(x))dW^k_s.
\end{equation*}
To get the existence for this equation, we define the approximating sequence $(\psi^n)_n$ of maps on $[T_1,2T_1]\times\mt^2$ by imposing
\begin{equation}
\psi^n_t(x)=\Phi_{T_1}(x)+\int^t_{T_1}\int_{\mt^2}K(\psi^n_s(x)-\psi^{n-1}_s(y))\, \xi_0(y)dy+\sum_k\int^t_{T_1}\sigma_k(\psi^{n-1}_s(x))dW^k_s\label{stocEulerflowiterdef}
\end{equation}
The existence, the continuity and the measure-preserving property for equation \eqref{stocEulerflowiterdef} are again not a direct consequence of Lemmata \ref{contlemma} and \ref{existlemma}, since here we start from $\Phi_{T_1}$ and not from the identity; here we also have the problem of the randomness of $\Phi_{T_1}$, which brings us to consider the strategy in the deterministic case. Following that strategy of the deterministic case, we can build $\psi^n$ and prove the continuity and the measure-preserving property. Then we apply the previous estimates, again with no chance in the constants and with the final time $T$ replaced by $T-T_1$ (again deterministic). This allows to conclude the existence on $[T_1,2T_1]$. The uniqueness on this interval is as in the step 1.

\textbf{Third step.} The regularity properties hold by Lemmma \ref{contlemma}, since $\Phi=G(\Phi)$ is in the image of $G$ ($G$ now being defined on the whole $[0,T]$).
\end{proof}

\section{The stochastic Euler vorticity equation}\label{stocEulervort_sec}

In this section we will prove Theorem \ref{mainvort}. First we need the existence of solutions to the stochastic Euler vorticity equation \eqref{simplItoform}.

\begin{proposition}
Let $\Phi$ be a solution to \eqref{stocEulerflow}. For $t\ge0$, define $\xi_t=(\Phi_t)_\#\xi_0$. Then $\xi$ has a density (still denoted by $\xi$) in $L^\infty([0,T]\times\mt^2\times\Omega)$, which is a distributional $L^\infty$ solution to the stochastic Euler equation \eqref{simplItoform}.
\end{proposition}

\begin{proof}
Fix $t>0$ and the probabilistic datum $\omega$ (omitted in the sequel). By Lemma \ref{abscont}, since $\Phi_t$ is measure preserving, $\xi_t$ is absolutely continuous with respect to the Lebesgue measure on $\mt^2$ and $\|\xi_t\|_{L^\infty}\le\|\xi_0\|_{L^\infty}$.

Let $\varphi$ be a test function, It\^o formula applied to $\varphi(\Phi_t)$ gives
\begin{eqnarray*}
d[\varphi(\Phi_t)]&=& u^\Phi_t(\Phi_t)\cdot\nabla\varphi(\Phi_t)\,dt+ \sum_k\sigma_k(\Phi_t)\cdot\nabla\varphi(\Phi_t)dW^k_r
\\&+& \frac12\tr[a(\Phi_t)D^2\varphi(\Phi_t)]\,dt.
\end{eqnarray*}
Now notice that, by definition of $\xi_t$, $u^\Phi_t=K*\xi_t$; so, integrating in $\xi_0dx$, we get \eqref{distribvort}.
\end{proof}

For the proof of the uniqueness, we will adapt a classical argument for the transport equation. We first recall the idea in the case $\sigma_k\equiv0$ for simplicity. A formal application of the chain rule gives
\begin{equation*}
\frac{d}{dt}\xi_t(\Phi_t)=\partial_t\xi_t(\Phi_t)+D\xi_t(\Phi_t)\frac{d\Phi_t}{dt}=(\partial_t\xi_t+u_t\cdot\nabla\xi_t)(\Phi_t)=0.
\end{equation*}
This implies that $\xi_t(\Phi_t)=\xi_0$, so that $\xi_t=\xi_0(\Phi_t^{-1})$ is completely determined by the flow. But we have used the chain rule for an object ($\xi_t$) which is not regular in general (and in fact there are counterexamples for irregular drifts). Thus we need to regularize $\xi$. This regularization $\xi^\eps$ solves a transport-type equation with an additional term, a commutator, which we need to control to conclude the argument. We use for this the argument in \cite{DiPLio}, \cite{Amb}, \cite{AmbCri}, where the commutator is an essential tool for the uniqueness of the transport equation.

First we need approximate identities. For this, let $\rho$ be a $C^\infty(\mr^2)$ nonnegative even function, with support in $[-1/2,1/2]^2$ and $\int_{\mr^2}\rho dx=1$. For $\eps>0$, define $\rho_\eps(x)=\eps^{-2}\rho(x/\eps)$. If $f$ is an integrable function on $\mt^2$, $f$ can be extended periodically to a locally integrable function on the whole $\mr^2$, so that the convolution $\rho_\eps *f$ makes sense and is still a $C^\infty$ periodic function.

For a vector field $v$ and a function $w$ on the torus, we define formally the commutator as
\begin{equation}
[v\cdot\nabla,\rho_\eps *]w:=v\cdot\nabla(\rho_\eps *w)-\rho_\eps *(v\cdot\nabla w).\label{formalcomm}
\end{equation}
Suppose that $v$ and $w$ are integrable and $v$ is divergence free. Then the expression above defines a measurable function on $\mt^2$. Indeed, the following equalities hold in distribution (the functions being thought as extended to the whole $\mr^2$):
\begin{equation}
\rho_\eps *(v\cdot\nabla w)=\rho_\eps *\diverg(vw)=-\int_{\mr^2}\nabla\rho_\eps(z)\cdot v(\cdot-z)w(\cdot-z)dz.\label{distribcomm}
\end{equation}
Besides, by \eqref{formalcomm} and \eqref{distribcomm}, the commutator reads
\begin{equation*}
[v\cdot\nabla,\rho_\eps *]w(x)=\int_{\mr^2}(v(x)-v(x-z))\cdot\nabla\rho_\eps(z)w(x-z)dz.
\end{equation*}
With the change of variable $y=z/\eps$, $x'=x'_\eps=x-\eps y$ we get
\begin{equation*}
[v\cdot\nabla,\rho_\eps *]w(x)= \int_{\mr^2}\frac{v(x'+\eps y)-v(x')}{\eps}\cdot\nabla\rho(y)w(x')dy.
\end{equation*}
If $v$ is in $W^{1,1}(\mt^2)$, then, for every $y$ in $\mr^2$, for a.e.\ $x'$ in $\mt^2$, $v(x'+\eps y)-v(x')=\eps\int^1_0Dv(x'+\xi\eps y)yd\xi$. Indeed, this is true for $v^\delta=\rho_\delta *v$ and, for fixed $y$, $v^\delta(x'+\eps y)-v^\delta(x')-\eps\int^1_0Dv^\delta(x'+\xi\eps y)yd\xi$, as function of $x'$, converges to $0$ a.e.\ as $\delta\rightarrow0$ (possibly passing to a subsequence). So, in this case, the commutator has the following expression:
\begin{equation}
[v\cdot\nabla,\rho_\eps *]w(x')= \int_{\mr^2}\int^1_0Dv(x'+\xi\eps y)yd\xi\cdot\nabla\rho(y)w(x')dy.\label{commrewritten}
\end{equation}

\begin{lemma}[Commutator lemma]
Let $p$ be in $[1,+\infty[$, let $v$ be in $W^{1,p}(\mt^2)$ with zero divergence, let $w$ be in $L^\infty(\mt^2)$. Then
\begin{equation*}
\lim_{\eps\rightarrow0}[v\cdot\nabla,\rho_\eps *]w=0\ \ \mbox{in }L^p(\mt^2)
\end{equation*}
and we have the inequality
\begin{equation*}
\|[v\cdot\nabla,\rho_\eps *]w\|_{L^p(\mt^2)}\le C\|Dv\|_{L^p(\mt^2)}\|w\|_{L^\infty(\mt^2)}.
\end{equation*}
\end{lemma}

\begin{proof}
The inequality follows integrating in $x$ the $p$-power of the expression on the LHS of \eqref{commrewritten}. Precisely, since $\rho$ is supported on $[-1/2,1/2]^2$, we have by H\"older inequality (remember $x'=x+\eps y$)
\begin{eqnarray*}
\lefteqn{\int_{\mt^2}|[v\cdot\nabla,\rho_\eps *]w|^pdx}\\
&\le& \int_{\mr^2}\int_{\mt^2}\int^1_0|Dv(x'+\xi\eps y)|^pd\xi|w(x')|^pdx'|y|^p|\nabla\rho(y)|^pdy \nonumber\\
&\le& \|Dv\|^p_{L^p(\mt^2)}\|w\|^p_{L^\infty(\mt^2)}\int_{\mr^2}|y|^p|\nabla\rho(y)|^pdy
\nonumber
\end{eqnarray*}
(the integral in $x'$ should be on $\mt^2-\eps y$, but by periodicity we can integrate on $\mt^2$ as well).

For the limit, it is enough to show that
\begin{equation*}
L^p(\mt^2)\mbox{-}\lim_{\eps\rightarrow0}[v\cdot\nabla,\rho_\eps *]w=w(\cdot)\left(\int_{\mr^2}Dv(\cdot)y\cdot\nabla\rho(y)dy\right).
\end{equation*}
Indeed, by the symmetry property of $\rho$, $\int_{\mr^2}y_i\partial_j\rho(y)dy=-C\delta_{ij}$ (where $C$ is independent of $i$) and so $\int_{\mr^2}Dv(x)y\cdot\nabla\rho(y)dy=-C\diverg w=0$. By \eqref{commrewritten} we have
\begin{eqnarray*}
\lefteqn{\int_{\mt^2}\left|[v\cdot\nabla,\rho_\eps *]w(x)-w(x)\left(\int_{\mr^2}Dv(x)y\cdot\nabla\rho(y)dy\right)\right|^pdx \le\int_{\mr^2}\int_{\mt^2}\int^1_0}\\
&& |w(x')Dv(x'+\xi\eps y)-w(x'+\eps y)Dv(x'+\eps y)|^pd\xi dx'|y|^p|\nabla\rho(y)|^pdy,
\end{eqnarray*}
hence it is enough to prove that
\begin{equation*}
\int_{\mt^2}\int^1_0|w(x')Dv(x'+\xi\eps y)-w(x'+\eps y)Dv(x'+\eps y)|^pd\xi dx'\rightarrow0
\end{equation*}
uniformly in $y$.
Using the continuity of translations in $L^p$ for the function $wDv$, we need only to show that
\begin{equation*}
\int_{\mt^2}\int^1_0|w(x')Dv(x'+\xi\eps y)-w(x')Dv(x')|^pd\xi dx'\rightarrow0.
\end{equation*}
Since $w$ is in $L^\infty$, this follows from $\int_{\mt^2}\int^1_0|Dv(x'+\xi\eps y)-Dv(x')|^pd\xi dx'\rightarrow0$, which is again a consequence of continuity of translation in $L^p$ applied to $Dv$.
\end{proof}

\begin{proposition}\label{reprEuler}
Let $\xi$ be a (distributional) $L^\infty$ solution to the stochastic Euler vorticity equation. Let $\Phi$ be a measure-preserving stochastic flow, which solves \eqref{stoclinearized} with $u=u^\xi$ (it exists by Corollary \ref{stocu}). Then $\xi_t=(\Phi_t)_\#\xi_0$.
\end{proposition}

\begin{proof}
We will prove that $\xi_t(\Phi_t)=\xi_0$ Lebesgue-a.e.. Having this, then, for every measurable bounded function $\varphi$ on $\mt^2$, $\lan\xi_t,\varphi\ran=\lan\xi_t(\Phi_t),\varphi(\Phi_t)\ran=\lan\xi_0,\varphi(\Phi_t)\ran$ (in the first equality we used the measure-preserving property) and so $\xi_t=(\Phi_t)_\#\xi_0$.

As mentioned before, we need to consider $\xi^\eps_t=\xi_t*\rho_\eps$ instead of $\xi_t$. Notice that, for every $x$, $\xi^\eps_t(x)=\lan\xi_t,\rho_\eps(x-\cdot)\ran$. So $\xi^\eps(x)$ is a progressively measurable process, with continuous trajectories, and the stochastic Euler vorticity equation, applied to the test function $\rho_\eps(x-\cdot)$, gives the following equality:
\begin{equation}
d\xi^\eps+ (u\cdot\nabla\xi)*\rho_\eps\,dt+ \sum_k(\sigma_k\cdot\nabla\xi)*\rho_\eps dW^k- \frac12\tr[aD^2\xi^\eps]\,dt =0,\label{xieps}
\end{equation}
which also reads
\begin{eqnarray*}
d\xi^\eps&+& u\cdot\nabla\xi^\eps\,dt+ \sum_k\sigma_k\cdot\nabla\xi^\eps dW^k- \frac12\tr[aD^2\xi^\eps]\,dt= [u\cdot\nabla,\rho_\eps *]\xi\,dt
\\
&+& \sum_k[\sigma_k\cdot\nabla,\rho_\eps *]\xi dW^k.
\end{eqnarray*}
Now, by \eqref{xieps}, since $\xi^\eps$ is adapted regular (together with $(u\cdot\nabla\xi)*\rho_\eps$, $\sigma_k\cdot\nabla\xi)*\rho_\eps$, $aD^2\xi^\eps$), we can apply It\^o-Kunita-Wentzell formula (see e.g.\ Theorem 8.3, page 188 of \cite{K}, with easy modifications for the case of an infinite number of $k$'s), obtaining for $\xi^\eps_t(\Phi_t)$
\begin{equation*}
d\xi^\eps_t(\Phi_t)= [u_t\cdot\nabla,\rho_\eps *]\xi_t(\Phi_t)\,dt+ \sum_k[\sigma_k\cdot\nabla,\rho_\eps *]\xi_t(\Phi_t) dW^k.
\end{equation*}
Since $\Phi$ is measure-preserving, integrating in space we get
\begin{eqnarray*}
&&E[\int_{\mt^2}|\xi^\eps_t(\Phi_t)-\xi_0|\,dx]\le \int^t_0\int_{\mt^2}E[|[u_r\cdot\nabla,\rho_\eps *]\xi_r|]dxdr\\
&+& \sum_k\int^t_0\int_{\mt^2}E[|[\sigma_k\cdot\nabla,\rho_\eps *]\xi_r|^2]^{1/2}dxdr.
\end{eqnarray*}

By the Commutator Lemma, for a.e.\ $r$ and $\omega$ in $\Omega$, $\int_{\mt^2}|[u_r\cdot\nabla,\rho_\eps *]\xi_r|\,dx$ tends to $0$ as $\eps\rightarrow0$. Besides, this term is dominated by
\begin{equation*}
C\|Du_r\|_{L^1(\mt^2)}\|\|\xi_r\|_{L^\infty(\mt^2)}\le C'\|\xi\|^2_{L^\infty([0,T]\times\mt^2\times\xi)}.
\end{equation*}
Indeed, for every $v$ in $L^\infty(\mt^2)$ and every finite $p\ge1$, $\|D(K*v)\|_{L^p(\mt^2)}\le C\|D^2(-\Delta)^{-1}v\|_{L^p(\mt^2)}\le C'\|v\|_{L^\infty(\mt^2)}$. So dominated convergence theorem gives that
\begin{equation*}
\lim_{\eps\rightarrow0}\int^t_0\int_{\mt^2}E[|[u_r\cdot\nabla,\rho_\eps *]\xi_r|]dxdr =0.
\end{equation*}

Similarly, for every $k$, for a.e.\ $r$ and $\omega$ in $\Omega$, $\int_{\mt^2}|[\sigma_kr\cdot\nabla,\rho_\eps *]\xi_r|^2dx$ tends to $0$ as $\eps\rightarrow0$ and is dominated by
\begin{equation*}
C\|D\sigma_k\|^2_{L^2(\mt^2)}\|\xi_r\|^2_{L^\infty(\mt^2)}.
\end{equation*}
Since $\sum_k\|D\sigma_k\|^2_{L^2(\mt^2)}\le \|\sum_k|D\sigma_k|^2\|_{L^\infty(\mt^2)}<+\infty$ by hypothesis, then we have (again by dominated convergence theorem)
\begin{equation*}
\lim_{\eps\rightarrow0} \sum_k\int^t_0\int_{\mt^2}E[|[\sigma_k\cdot\nabla,\rho_\eps *]\xi_r|^2]dxdr=0 .
\end{equation*}


Thus, for any fixed $t>0$, $\xi^\eps_t(\Phi_t)$ tends to $\xi_0$ in $L^1(\mt^2\times\Omega)$ as $\eps\rightarrow0$. Since $\xi^\eps_t$ converges to $\xi_t$ in $L^1(\mt^2\times\Omega)$ (the convergence in $L^1(\mt^2)$ being dominated by $\|\xi\|_{L^\infty}$) and $\Phi_t$ is measure-preserving, $\xi_t^\eps(\Phi_t)$ converges to $\xi_t(\Phi_t)$ in $L^1(\mt^2\times\Omega)$ and thus $\xi_t(\Phi_t)=\xi_0$, which is our thesis.
%
\end{proof}

\begin{corollary}
The uniqueness for the stochastic Euler vorticity equation (in the class of $L^\infty$ solutions) holds.
\end{corollary}

\begin{proof}
The above Proposition \ref{reprEuler} tells that a solution $\xi$ to the stochastic Euler vorticity equation is completely determined by the associated flow $\Phi$ which solves \eqref{stoclinearized} with $u=u^\xi$; again for the proposition, $u=u^\Phi$ and so $\Phi$ solves (\ref{stocEulerflow}). Thus the uniqueness for \eqref{stocEulerflow} implies the uniqueness for the stochastic Euler vorticity equation.
\end{proof}

This concludes the proof of Theorem \ref{mainvort}.

\section{Stability}

In this section we want to prove a stability result, both at Lagrangian and Eulerian points of view, when the kernel $K$ is regularized.

Precisely, take a family $(\rho_\eps)_\eps$ of even compactly supported resolutions of identity and define $K^\eps:=K*\rho_\eps$. Consider the approximated non-local ODE
\begin{equation}
\Phi^\eps_t(x)= x +\int^t_0\int_{\mt^2}K^\eps(\Phi^\eps_r(x)-\Phi^\eps_r(y)\, \xi_0(y)dy +\sum^\infty_{k=1}\int^t_0\int_{\mt^2}\sigma_k(\Phi^\eps_r(x))dW^k_r\label{KepsSDE}
\end{equation}
and the approximated stochastic Euler vorticity equation
\begin{equation}
d\xi^\eps +u^{\eps,\xi^\eps}\cdot\nabla\xi^\eps\,dt +\sum_k\sigma_k\cdot\nabla\xi^\eps dW^k = \frac12 C\Delta\xi^\eps,\label{KepsSPDE}
\end{equation}
where $u^{\eps,\xi^\eps}:=K^\eps *\xi^\eps$.

One can repeat all the previous definitions and arguments with $K^\eps$ in place of $K$, to get the analogues of Theorem \ref{mainflow} and Theorem \ref{mainvort}: there exists a unique measure-preserving stochastic continuous flow $\Phi$ solving \eqref{KepsSDE}, which is also $C^\alpha$ in space, for every $\alpha<1$ and $C^\beta$ in time, for every $\beta<1/2$; there exists a unique $L^\infty$ distributional solution $\xi^\eps$ for \eqref{KepsSPDE}. Moreover it holds
\begin{equation}
\xi^\eps_t=(\Phi^\eps_t)_\#\xi_0.\label{repreps}
\end{equation}

The first stability result is for flows:

\begin{proposition}
The family $(\Phi^\eps)_\eps$ converges to $\Phi$ (as $\eps\rightarrow0$) in $C([0,T];L^1(\mt^2\times\Omega))$.
\end{proposition}

\begin{proof}
The fact that $\Phi^\eps$ and $\Phi$ belong to $C([0,T];L^1(\mt^2\times\Omega))$ can be proved easily, using similar techniques to those below. For the convergence, call $Z^\eps_t(x)=\Phi^\eps_t(x)-\Phi_t(x)$. As in the proof of Lemma \ref{stockeylemma}, we would like to apply It\^o formula for $|Z^\eps|$. Proceeding as in that proof (applying It\^o formula to $f_\delta(x)=(|x|^2+\delta)^{1/2}$), we get
\begin{eqnarray*}
E|Z^\eps_t(x)| &\leq& \int^t_0\int_{\mt^2}E\left|K^\eps(\Phi^\eps_r(x)-\Phi^\eps_r(y))-K(\Phi_r(x)-\Phi_r(y))\right||\xi_0(y)|dydr\\
&+&2L_\sigma^2\int^t_0E|Z^\eps_r(x)|dr.
\end{eqnarray*}
Integrating this inequality in $x$, since $\xi_0$ is bounded, we obtain
\begin{eqnarray}
\lefteqn{\int_{\mt^2}E|Z^\eps_t(x)|\,dx}\nonumber\\
&\le& \|\xi_0\|_{L^\infty}\int^t_0\int_{\mt^2}\int_{\mt^2}E\left|K^\eps(\Phi^\eps_r(x)-\Phi^\eps_r(y))-K(\Phi_r(x)-\Phi_r(y))\right|\,dxdydr\nonumber\\
& & +2L_\sigma^2\int^t_0\int_{\mt^2}E|Z^\eps_r(x)|\,dxdr\nonumber\\
&\le& \|\xi_0\|_{L^\infty}\int^t_0\int_{\mt^2}\int_{\mt^2}E\left|K^\eps(\Phi^\eps_r(x)-\Phi^\eps_r(y))-K(\Phi^\eps_r(x)-\Phi^\eps_r(y))\right|\,dxdydr \nonumber\\
& & +\|\xi_0\|_{L^\infty}\int^t_0\int_{\mt^2}\int_{\mt^2}E\left|K(\Phi^\eps_r(x)-\Phi^\eps_r(y))-K(\Phi_r(x)-\Phi_r(y))\right|\,dxdydr\nonumber\\
& & +2L_\sigma^2\int^t_0\int_{\mt^2}E|Z^\eps_r(x)|\,dxdr.\label{estapprox}
\end{eqnarray}
For the first integral of \eqref{estapprox}, we exploit the fact that $\Phi^\eps$ is measure-preserving, for every $\eps$; so we have
\begin{eqnarray*}
&&\hspace{-2truecm}\lefteqn{\int^t_0\int_{\mt^2}\int_{\mt^2}E\left|K^\eps(\Phi^\eps_r(x)-\Phi^\eps_r(y))-K(\Phi^\eps_r(x)-\Phi^\eps_r(y))\right|\,dxdydr}\\
&=& \int^t_0\int_{\mt^2}\int_{\mt^2}E\left|K^\eps(x-y)-K(x-y)\right|\,dxdydr\\
&\le& T\int_{\mt^2}|K^\eps(x')-K(x')|\,dx',
\end{eqnarray*}
where we have used, in the last passage, the change of variable $x-y=x'$, $x+y=y'$ (this implies a change of domain, but the $L^1$ norm of $K^\eps(x')-K(x')$ on the new domain is comparable with that on the torus). For the second integral of \eqref{estapprox}, we exploit the log-Lipschitz property of $K$ (estimate \eqref{keyest}) and get
\begin{eqnarray*}
\lefteqn{\int^t_0\int_{\mt^2}\int_{\mt^2}E\left|K(\Phi^\eps_r(x)-\Phi^\eps_r(y))-K(\Phi_r(x)-\Phi_r(y))\right|\,dxdydr}\\
&\le& L_K\int^t_0\int_{\mt^2}\int_{\mt^2}E\gamma(|Z^\eps_r(x)-Z^\eps_r(y)|)dxdydr\\
&\le& L_K\int^t_0\int_{\mt^2}\int_{\mt^2}E\left[\gamma(|Z^\eps_r(x)|)+\gamma(|Z^\eps_r(y)|)\right]dxdydr\\
&\le& 2L_K\int^t_0\int_{\mt^2}\gamma(E|Z^\eps_r(x)|)dxdr,
\end{eqnarray*}
where we have used the sub-additivity of $\gamma$ ($\gamma(|x+y|)\le \gamma(|x|)+\gamma(|y|)$, as it can be easily checked) and Jensen inequality. Putting all together, we have
\begin{equation*}
\int_{\mt^2}E|Z^\eps_t(x)|\,dx\le T\|K^\eps-K\|_{L^1(\mt^2)} +(2L_K\|\xi_0\|_{L^\infty}+2L_\sigma^2)\int^t_0\int_{\mt^2}\gamma(E|Z^\eps_r(x)|)dxdr.
\end{equation*}
Again by comparison, we get $\int_{\mt^2}E|Z^\eps_t(x)|\,dx\le z^{2L_K\|\xi_0\|_{L^\infty}+2L_\sigma^2}(t,T\|K^\eps-K\|_{L^1(\mt^2)})$, where $z$ is defined as in \eqref{z}. Since $K$ is in $L^1(\mt^2)$, $\|K^\eps-K\|_{L^1(\mt^2)}$ tends to $0$ (as $\eps\rightarrow0$), so
\begin{equation*}
\sup_{t\in[0,T]}\int_{\mt^2}E|Z^\eps_t(x)|\,dx\le \sup_{t\in[0,T]}z(t,T\|K^\eps-K\|_{L^1(\mt^2)})\rightarrow 0.
\end{equation*}
The proof is complete.
\end{proof}

Here is the result for the vorticity:

\begin{proposition}
The family $(\xi^\eps)_\eps$ converges weakly to $\xi$ (as $\eps\rightarrow0$), in the following sense. For every $\varphi$ in $C_b(\mt^2)$,
\begin{equation*}
E\left|\int_{\mt^2}\varphi\xi^\eps_tdx-\int_{\mt^2}\varphi\xi_tdx\right|\rightarrow0
\end{equation*}
for every $t$ and in $L^p([0,T])$, for any $p \in [1,\infty)$.
\end{proposition}

\begin{proof}
First, notice that, by \eqref{repreps},
\begin{equation*}
\int_{\mt^2}\varphi\xi^\eps_tdx =\int_{\mt^2}\varphi(\Phi^\eps_t)\, \xi_0dx
\end{equation*}
and the same without $\eps$. In particular, $\varphi(\Phi^\eps_t)\, \xi_0$ is dominated a.e.\ by a constant. Now fix the time $t$. We use here a classical argument in measure theory. Suppose by contradiction that there exist $\delta>0$ and a sequence $\eps_n\rightarrow0$ such that
\begin{equation}
E\left|\int_{\mt^2}\varphi(\Phi^{\eps_n}_t)\, \xi_0dx-\int_{\mt^2}\varphi(\Phi_t)\, \xi_0dx\right|\ge\delta.\label{hpcontr}
\end{equation}
The previous proposition gives that $\Phi^{\eps_n}_t$ converges to $\Phi_t$ in $L^1(\mt^2\times\Omega)$. So we have for a subsequence $\eps_{n_k}$ that $\Phi^{\eps_{n_k}}_t$ tends to $\Phi_t$ for a.e.\ $(x,\omega)$ and similarly for $\varphi(\Phi^{\eps_{n_k}}_t)$, since $\varphi$ is continuous. Hence, by dominated convergence theorem, we get that
\begin{equation*}
E\left|\int_{\mt^2}\varphi(\Phi^{\eps_{n_k}}_t)\, \xi_0dx-\int_{\mt^2}\varphi(\Phi_t)\, \xi_0dx\right|\rightarrow0,
\end{equation*}
which contradicts \eqref{hpcontr}. We have proved convergence at $t$ fixed. Convergence in $L^p([0,T])$, for any finite $p$, follows from this result and the Lebesgue Dominated Convergence Theorem.
\end{proof}

\section{An alternative way: reduction to the deterministic case}

In this section we will see how to deduce the results in the stochastic case by a suitable transformation, assuming the deterministic case and more regularity for the $\sigma_k$'s. As we already said, we will not develop this method in all the details.

At a Lagrangian level (trajectories), consider the SDE with only the stochastic integral, namely
\begin{equation}
d\psi=\sum_k\sigma_k(\psi)\circ dW^k.\label{SDEstoch}
\end{equation}
It is well known that, if the fields $\sigma_k$'s are regular enough ($C^3$ should be sufficient, $C^2$ is assumed in every ``classical'' result) and divergence-free, then there exists a stochastic flows $\psi$ of $C^{1,1}(\mt^2)$ measure-preserving diffeomorphisms solving \eqref{SDEstoch} (a $C^{1,1}$ diffeomorphism is a $C^1$ map with Lipschitz-continuous derivatives, together with its inverse). The inverse flow $\psi_t^{-1}$ satisfies
\begin{equation*}
d\psi^{-1}_t(x)=-\sum_k\sigma_k(x)\cdot\nabla\psi^{-1}_t(x)\circ dW^k.
\end{equation*}

Now let $\Phi$ be the Euler stochastic flow (solving \eqref{stocEulerflow}) and make a change of variable, composing with $\psi_t^{-1}$: call
\begin{equation}
\tl{\Phi}(t,x,\omega)=\psi_{t,\omega}^{-1}(\Phi_{t,\omega}(x)).\label{changevar}
\end{equation}
Using the It\^o-Kunita-Wentzell formula, we obtain the following random ODE for $\tl{\Phi}$:
\begin{equation*}
d\tl{\Phi}_t=(D\psi_t)^{-1}u^\Phi_t(\psi_t(\tl{\Phi}_t))\,dt,
\end{equation*}
where $u^\Phi$ is as in \eqref{expr_u}. This equation reads also as
\begin{equation}
d\tl{\Phi}_t=\tl{u}^{\tl{\Phi}}_t(\tl{\Phi}_t)\,dt,\label{randomODE}
\end{equation}
where
\begin{equation*}
\tl{u}^{\tl{\Phi}}(t,x,\omega)=(D\psi_{t,\omega}(x))^{-1}\int_{\mt^2}K(\psi_{t,\omega}(x)-\psi_{t,\omega}(\tl{\Phi}_{t,\omega}(y)))\, \xi_0(y)dy.
\end{equation*}
The equation \eqref{randomODE} is not \eqref{Eulerflow}, but the drift $\tl{u}^{\tl{\xi}}$ has the same regularity properties of the drift $u^\Phi$ of \eqref{Eulerflow}, provided $\psi$ is a flow of $C^{1,1}(\mt^2)$ diffeomorphisms, since the term $D\psi_t$ appears; here we need $\sigma$ to be at least $C^2$. Thus, one could proceed as follows:
\begin{enumerate}
\item first we can repeat the argument in the deterministic part, to get the existence and the uniqueness for $\tl{\Phi}$ satisfying \eqref{randomODE}; since $\psi$ is a regular flow adapted to the Brownian filtration, this implies the strong existence and the strong uniqueness for $\Phi$ itself (plus the homeomorphism property), i.e.\ Theorem \ref{mainflow};
\item then Section \ref{stocEulervort_sec} applies and we deduce Theorem \ref{mainvort}.
\end{enumerate}

This can be seen also at an Eulerian level (velocity field). Heuristically, with the change of variable \eqref{changevar}, we should consider, as new vorticity, $\tl{\xi}_t=\xi_0(\tl{\Phi}_t^{-1})=\xi_t(\psi_t)$. Indeed, let $\xi$ be a solution to \eqref{Euler Strat} and let $\psi$ be as above, call
\begin{equation*}
\tl{\xi}(t,x,\omega)=\xi(t,\psi(t,x,\omega),\omega).
\end{equation*}
Applying, this time formally, the It\^o-Kunita-Wentzell formula, we obtain the following random PDE for $\tl{\xi}$:
\begin{equation}
\partial_t\tl{\xi}+\tl{u}^{\tl{\xi}}\cdot\nabla\tl{\xi}=0,\label{randomPDE}
\end{equation}
where
\begin{equation*}
\tl{u}^{\tl{\xi}}=(D\psi_{t,\omega}(x))^{-1}\int_{\mt^2}K(\psi_{t,\omega}(x)-\psi_{t,\omega}(y))\tl{\xi}_t(y)dy.
\end{equation*}
This fact, as well as its converse (the passage from $\tl{\xi}$ to $\xi$), can be made rigorous in the following way. First, we take $\xi^\eps=\xi*\rho_\eps$ (where $\rho_\eps$ are even compactly supported mollifiers) and write the equation for $\xi^\eps$ (using commutators only for the $\sigma_k$'s):
\begin{equation*}
\partial_t\xi^\eps+(u^\xi\cdot\nabla\xi)^\eps+\sum_k\sigma_k\cdot\nabla\xi^\eps\circ\dot{W}^k-\sum_k[\sigma_k\cdot\nabla,\rho_\eps*]\xi\circ\dot{W}^k=0.
\end{equation*}
Then we multiply this equation by $\varphi(\psi^{-1})$, where $\varphi$ is any regular test function on $\mt^2$. In this way we obtain \eqref{randomPDE} for $\xi^\eps(\psi)$, with $(u^\xi\cdot\nabla\xi)^\eps(\psi)$ in place of $\tl{u}^{\tl{\xi}}\cdot\nabla\tl{\xi}$  and with the additional commutator term $\sum_k[\sigma_k\cdot\nabla,\rho_\eps*]\xi(\psi)\circ\dot{W}^k$. Finally we let $\eps$ go to $0$, getting \eqref{randomPDE}.

Again \eqref{randomPDE} is not the deterministic Euler vorticity equation (\eqref{Euler Strat} with $\sigma=0$), but its drift $\tl{u}^{\tl{\xi}}$ has the same regularity properties of the drift $u^\xi$ of \eqref{Eulerflow}, provided $\psi$ is a flow $C^{1,1}(\mt^2)$ diffeomorphisms. So one can repeat the arguments in the deterministic case (flows and commutator lemma), to get the existence and the uniqueness for the random PDE \eqref{randomPDE}, then the strong existence and the strong uniqueness for \ref{Euler Strat} follow immediately.

Finally we mention that the passage between $\xi$ and $\tl{\xi}$ can be seen at a more abstract level; this is a classical remark, due at least to Lamperti, Doss and Sussmann (\cite{Lam}, \cite{Dos}, \cite{Sus}). Suppose to have an SPDE of the form
\begin{equation*}
d\xi+A(\xi)\, \xi\,dt+\sum_kB_k\xi\circ dW^k=0,
\end{equation*}
where $A(x)$ and $B_k$ are linear operators (for simplicity assume $B_k$ time-independent); in our case, $A(\xi)=u^\xi\cdot\nabla$ and $B_k=\sigma_k\cdot\nabla$. Consider formally
\begin{equation*}
\tl{\xi}_t=e^{\sum_kB_kW^k_t}\xi_t;
\end{equation*}
in our case, this corresponds to the composition $\xi(\psi)$. Then formally $\tl{\xi}$ satisfies the following random PDE:
\begin{equation*}
\partial_t\tl{\xi}+e^{\sum_kB_kW^k_t}A(e^{-\sum_kB_kW^k_t}\tl{\xi})e^{-\sum_kB_kW^k_t}\tl{\xi}=0.
\end{equation*}
Thus we have reduced an SPDE to a random PDE, which can be treated through deterministic techniques.

\appendix

\section{A useful inequality}

This section contains a proof of an auxiliary inequality used in a crucial way twice in our paper.
\begin{lemma}\label{lem-ineq} Assume that $A,B>0$ and $T>0$.
Suppose that $(\rho_n)_{n=0}^\infty$ is a sequence of continuous nonnegative functions defined on the interval $[0,T]$ such that for every $\eps\in (0,1)$ and every $n$,
\begin{equation}\label{ineq-01}
\rho^n_t\leq A \log \frac1\eps\int^t_0\rho^{n-1}_s\,ds + \eps Bt,\;\;  t\in [0,T].
\end{equation}
Then
\begin{eqnarray}
\nonumber \rho^n_t &\leq&  \frac{(At)^n}{\sqrt{2\pi n}} \sup_{s\in [0,t]}\vert \rho^0_{s}\vert  + B  t (e^{At-1})^n
,\;\;\; t\in [0,T].
\label{ineq-02}
\end{eqnarray}
\end{lemma}
\begin{proof}[Proof of Lemma \ref{lem-ineq}]
By Induction one can show that for every $n\in\mathbb{N}^\ast$ and every $\eps\in (0,1)$
\begin{eqnarray}\nonumber
\rho^n_t &\leq& (A(-\log\eps))^n\int^t_0\ldots\int^{s_2}_0\rho^0_{s_1}ds_1\ldots ds_n \\
\nonumber
&+& B\eps t\sum_{k=0}^{n-1}(A(-\log\eps))^k \int^t_0\ldots\int^{s_2}_0ds_1\ldots ds_k
\\&\leq &\frac{(A(-\log\eps) t)^n}{n!} \sup_{s\in [0,t]}\vert \rho^0_{s}\vert  + B\eps t\sum_{k=0}^{n-1}\frac{(A(-\log\eps)  t)^k}{k!},\;\;\; t\in [0,T].
\label{ineq-03}
\end{eqnarray}
Let us take $n\in\mathbb{N}$. Choose $\eps=e^{-n}$. Then by the above inequality and Stirling's inequality,
\begin{eqnarray}
\nonumber \rho^n_t &\leq& \frac{(An t)^n}{n!} \sup_{s\in [0,t]}\vert \rho^0_{s}\vert  + B e^{-n} t\sum_{k=0}^{n-1}\frac{(An  t)^k}{k!}
\\
\nonumber
&\leq& \frac{(eAn t)^n}{n!} \sup_{s\in [0,t]}\vert \rho^0_{s}\vert  + B  t (e^{At-1})^n\\
&\leq& \frac{(eAt)^n}{\sqrt{2\pi n}} \sup_{s\in [0,t]}\vert \rho^0_{s}\vert  + B  t (e^{At-1})^n
,\;\;\; t\in [0,T].
\label{ineq-04}
\end{eqnarray}
This concludes the proof.
\end{proof}
\begin{corollary}In the framework of the above Lemmma, if $eAT^\ast<1$, then $\sup_{t\in [0,T^\ast]}\rho^n_t \to 0$.
\end{corollary}
\begin{proof}
If $eAT^\ast<1$, then $\sup_{t\in [0,T^\ast]}\rho^n_t$ is bounded from above by a sum of the $n$-th terms of two convergent geometrical series.
\end{proof}

\section{Proof of inequality \eqref{keyest}}

We give a sketch of the proof of inequality \eqref{keyest}. Call $G$ is the Green function of the Laplace operator $-\Delta$ on the torus $\mt^2=[-1/2,1/2]^2$ (with periodic boundary condition). We will prove:

\begin{proposition}\label{Green}
The function $G$ is in $C^\infty(\mt^2\setminus\{0\})$. Its behaviour in $0$ is given by
\begin{equation*}
|G(x)|\le C(-\log|x|+1)
\end{equation*}
and that of its derivative $D^{(n)}$, $n$ positive integer, by
\begin{equation*}
|D^nG(x)|\le C_n(|x|^{-n}+1).
\end{equation*}
\end{proposition}

Assuming this result, we get that $|K(x)|\le C_1(|x|^{-1}+1)$. This implies the estimate \eqref{keyest} by an elementary argument (see \cite{MarPul}, Appendix 2.3).

Proposition \ref{Green} is a special case (at least for $n\le 2$) of a general fact, valid for compact $C^\infty$ Riemannian manifolds of finite dimensions, see \cite[section 4.2]{Aub}, for the statement and a proof. We give here a different proof, taken in spirit from \cite{BatRei} (which studies the 3D case).

\begin{proof}[Sketch of the proof]
It is easy to see that the Fourier expansion of $G$ is
\begin{equation*}
G(x)=-\frac{1}{4\pi^2}\sum_{k\in\mathbb{Z}^2,k\neq0}\frac{1}{|k|^2}e^{2\pi ik\cdot x}
\end{equation*}
Since this expression seems not helpful in the analysis of regularity around $0$, we will use the solution $v$, in $L^2([0,T]\times\mt^2)$, of the heat equation
\begin{equation*}
\partial_tv=\Delta v,
\end{equation*}
with initial condition $v_0=\delta_0-1$ (more precisely, $v_t\rightharpoonup\delta_0-1$ as $t\rightarrow0$). It is easy to see that this unique solution can be expressed in two ways: one with its Fourier expansion, which is
\begin{equation}
v(t,x)=\sum_{k\in\mathbb{Z}^2,k\neq0}e^{-4\pi^2|k|^2t}e^{2\pi ik\cdot x},\label{vFourier}
\end{equation}
the other with Gaussian densities, that is
\begin{equation}
v(t,x)=-1+\frac{1}{4\pi t}\sum_{l\in\mathbb{Z}^2}\exp{\frac{-|x-l|^2}{4t}}.\label{vGauss}
\end{equation}
One verifies, e.g.\ using \eqref{vFourier}, that
\begin{eqnarray*}
G(x)&=&-\int^{+\infty}_0v(t,x)\,dt=-\int^{+\infty}_1v(t,x)\,dt-\int^1_0v(t,x)\,dt\\&=:&-G_1(x)-G_2(x).
\end{eqnarray*}
Now $G_1$ is in $C^\infty(\mt^2)$, as one can see from its Fourier expansion, again from \eqref{vFourier}. For $G_2$ we exploit \eqref{vGauss}:
\begin{eqnarray*}
G_2(x)&=&\left(-1+\int^1_0\frac{1}{4\pi t}\sum_{l\in\mathbb{Z}^2,l\neq0}\exp{\frac{-|x-l|^2}{4t}}dt\right)
\\&+&\int^1_0\frac{1}{4\pi t}\exp{\frac{-|x|^2}{4t}}=:G_3(x)+G_4(x),
\end{eqnarray*}
the sum being between functions on $\mr^2$ (though $x$ is still in $[-1/2,1/2]^2$). The first addend $G_3$ is $C^\infty$ on an open neighborhood of $[-1/2,1/2]^2$ (e.g.\ $]-3/4,3/4[^2$): indeed, for any $n$ nonnegative integer, we have
\begin{eqnarray*}
&&\hspace{-1truecm}\lefteqn{\int^1_0\left|D^{(n)}\frac{1}{4\pi t}\sum_{l\in\mathbb{Z}^2,l\neq0}\exp{\frac{-|x-l|^2}{4t}}\right|dt}\\
&\lesssim & \int^1_0t^{-(2n+1)}\sum_{l\neq0}\exp{\frac{-|x-l|^2}{4t}}dt \\
& \lesssim& \int^1_0t^{-(2n+1)}\sum^\infty_{h=1}\exp{\frac{-h}{ct}}dt\\
&\sim& \int^1_0t^{-(2n+1)}e^{-1/(ct)}dt<+\infty,
\end{eqnarray*}
for some $c>0$ independent of $x$, when $x$ is in $]-3/4,3/4[^2$. The second addend $G_4$ is in $C^\infty(]-3/4,3/4[^2\setminus\{0\})$. So $G$ is in $C^\infty(\mt^2\setminus\{0\}$. For the behaviour in $0$, this is given by the behaviour of $G_4$, which is computed by standard techniques. We have, with the change of variable $s=|x|^{-1/2}t$,
\begin{equation*}
G_4(x)\sim \int^{|x|^{-1/2}}_0s^{-1}e^{-1/(4s)}ds\sim -\log|x|
\end{equation*}
and, for $n\ge1$,
\begin{equation*}
|D^{(n)}G_4(x)|\sim |x|^{-n}\int^{|x|^{-1/2}}_0s^{-(2n+1)}e^{-1/(4s)}ds\sim |x|^{-n}.
\end{equation*}
The proof is complete.
\end{proof}


\begin{thebibliography}{99}

\bibitem{Amb}L. Ambrosio, Transport equation and Cauchy problem for $BV$ vector fields, \textit{Invent. Math.} \textbf{158} (2004), no. 2, 227--260.

\bibitem{AmbCri}L. Ambrosio, G. Crippa, Existence, uniqueness, stability and differentiability properties of the flow associated to weakly differentiable vector fields, in Transport equations and multi-D hyperbolic conservation laws, Lect. Notes of the Unione Matematica Italiana, 2008, Volume 5, I, 3--57.

\bibitem{Aub}T. Aubin, \textit{Some nonlinear problems in Riemannian geometry}, Springer-Verlag, Berlin, 1998.

\bibitem{BatRei}J. Batt, G. Rein, A rigorous stability result for the Vlasov-Poisson system in three dimensions, \textit{Ann. Mat. Pura Appl. (4)} \textbf{164} (1993), 133--154.

\bibitem{BaxHar} P. Baxendale, T.E. Harris, Isotropic stochastic flows, \textit{Ann. Probab.} \textbf{14} (1986), no. 4, 1155--1179.

\bibitem {BeGaKu} D. Bernard, K. Gaw\k{e}dzki, A. Kupiainen, Anomalous scaling
in the N-point functions of a passive scalar, \textit{Phys. Rev. E} (3)
\textbf{54} (1996), no. 3, 2564--2572.

\bibitem {Bes1}H. Bessaih, Martingale solutions for stochastic Euler
equations, \textit{Stochastic Anal. Appl.} \textbf{17} (1999), no. 5, 713--725.

\bibitem {Bes2}H. Bessaih, Stochastic weak attractor for a dissipative Euler
equation, \textit{Electron. J. Probab.} \textbf{5} (2000), no. 3, 16 pp.

\bibitem {Bes3}H. Bessaih, Stationary solutions for the 2D stochastic
dissipative Euler equation. Seminar on Stochastic Analysis, Random Fields and
Applications V, 23--36, Progr. Probab., 59, Birkh\"{a}user, Basel, 2008.

\bibitem {BesFla}H. Bessaih, F. Flandoli, 2-D Euler equation perturbed by
noise, \textit{NoDEA Nonlinear Differential Equations Appl.} \textbf{6}
(1999), no. 1, 35--54.

\bibitem{brzgolond} Z. Brze\'zniak, B. Goldys, M.  Ondrej\'at, \textit{Stochastic geometric partial differential equations}, in {\sc New trends in stochastic analysis and related topics}, 1-32, Interdiscip. Math.. Sci., \textbf{12}, World Sci. Publ., Hackensack, NJ, 2012.

\bibitem {BrzPes}Z. Brze\'{z}niak, S. Peszat, Stochastic two dimensional Euler
equations, \textit{Ann. Probab.} \textbf{29} (2001), no. 4, 1796--1832.

\bibitem {CapCut}M. Capi\'{n}ski, N. J. Cutland, Stochastic Euler equations on
the torus, \textit{Ann. Appl. Probab.} \textbf{9} (1999), no. 3, 688--705.

\bibitem {CeVi}A. Celani, D. Vincenzi, Intermittency in passive scalar decay,
\textit{Phys. D} \textbf{172} (2002), no. 1-4, 103--110.

\bibitem{Chemin} J.Y. Chemin, \`Equations d'Euler d'un fluide incompressible, \textit{Facettes math\`ematiques de la m\`ecanique des fluides}, 9-30, Ed. \'Ec. Polytech., Palaiseau, 2010.

\bibitem{DeL+S} C. De Lellis, L. Sz\`ekelyhidi, Jr.,  The Euler equations as a differential inclusion, \textit{Ann. of Math. (2)} \textbf{170} (2009), no. 3, 1417--1436.

\bibitem {DiPLio}R. J. DiPerna, P.-L. Lions, Ordinary differential equations, transport theory and Sobolev spaces, \textit{Invent. Math.} \textbf{98} (1989), no. 3, 511--547.

\bibitem {Dos}H. Doss, Liens entre \'euations diff\'erentielles stochastiques et ordinaires, \textit{Ann. Inst. H. Poincaré Sect. B (N.S.)} \textbf{13} (1977), no. 2, 99--125.

\bibitem {FaGaVe}G. Falkovich, K. Gaw\k{e}dzki, M. Vergassola, Particles and
fields in fluid turbulence, \textit{Rev. Modern Phys.} \textbf{73} (2001), no.
4, 913--975.

\bibitem {Fla}F. Flandoli, \textit{Random Perturbation of PDEs and Fluid
Dynamic Models}, Saint Flour summer school lectures 2010, Lecture Notes in
Math. \textbf{2015}, Springer, Berlin 2011.

\bibitem {FGPEuler}F. Flandoli, M. Gubinelli, E. Priola, E. Full
well-posedness of point vortex dynamics corresponding to stochastic 2D Euler
equations, \textit{Stochastic Process. Appl.} \textbf{121} (2011), no. 7, 1445--1463.

\bibitem {Gaw}K. Gaw\k{e}dzki, Stochastic processes in turbulent transport, arXiv:0806.1949v2.

\bibitem{GlaSveVic}N. Glatt-Holtz, V. \v{S}ver\'{a}k, V. Vicol, On Inviscid Limits for the Stochastic Navier-Stokes Equations and Related Models, arXiv:1302.0542.

\bibitem{GH+V} N. Glatt-Holtz and V.C. Vicol,  Local and global existence of smooth solutions for the stochastic Euler equations with multiplicative noise, \textit{Ann. Probab.} \textbf{42} (2014), no. 1, 80--145.

\bibitem {Kim1}J. U. Kim, On the stochastic Euler equations in a
two-dimensional domain,\textit{ SIAM J. Math. Anal.} \textbf{33} (2002), no.
5, 1211--1227.

\bibitem {Kim2}J. U. Kim, Existence of a local smooth solution in probability
to the stochastic Euler equations in $\mathbb{R}^{3}$, \textit{J. Funct.
Anal.} \textbf{256} (2009), no. 11, 3660--3687.

\bibitem {Kra1}R. H. Kraichnan, Small-scale structure of a scalar field convected by turbulence, \textit{Phys. Fluids} \textbf{11} (1968), 945--963.

\bibitem {Kra2}R. H. Kraichnan, Anomalous scaling of a randomly advected passive scalar, \textit{Phys. Rev. Lett.} \textbf{72} (1994), 1016--1019.

\bibitem {K} H. Kunita, \textit{Stochastic differential equations and
stochastic flows of diffeomorphisms}, Ecole d'\'{e}t\'{e} de probabilit\'{e}s
de Saint-Flour, XII---1982, 143-303, Lecture Notes in Math. \textbf{1097},
Springer, Berlin, 1984.

\bibitem {Kun2} H. Kunita, \textit{Stochastic Flows and Stochastic Differential Equations}, Cambridge Studies in Advanced Math. \textbf{24}. Cambridge University Press, Cambridge, 1997.

\bibitem {KuMu}A. Kupiainen, P. Muratore-Ginanneschi, Scaling, renormalization
and statistical conservation laws in the Kraichnan model of turbulent
advection, \textit{J. Stat. Phys.} \textbf{126} (2007), no. 3, 669--724.

\bibitem {Lam}J. Lamperti, A simple construction of certain diffusion porcesses, \textit{J. Math. Kyoto Univ.} \textbf{4} (1964), 161--170.

\bibitem {LeRa}Y. Le Jan, O. Raimond, Integration of Brownian vector fields,
\textit{Ann. Probab.} \textbf{30} (2002), no. 2, 826--873.

\bibitem{Majda+Bertozzi} A.J. Majda, A.L. Bertozzi, Vorticity and incompressible flow, \textit{Cambridge Texts in Applied Mathematics}, Cambridge University Press, Cambridge, 2002.

\bibitem {MarPul}C. Marchioro, M. Pulvirenti, \textit{Mathematical Theory of
Incompressible Nonviscous Fluids}, Springer, Berlin 1994.

\bibitem {MikRoz}R. Mikulevicius, B. L. Rozovskii, Stochastic Navier-Stokes
equations for turbulent flows, \textit{SIAM J. Math. Anal.} \textbf{35}
(2004), no. 5, 1250--1310.

\bibitem {MikVal1}R. Mikulevi\v{c}ius, G. Valiukevi\v{c}ius, On stochastic
Euler equation, \textit{Liet. Mat. Rink.} \textbf{38} (1998), no. 2, 234--247;
translation in \textit{Lithuanian Math. J.} \textbf{38} (1998), no. 2,
181--192 (1999)

\bibitem {MikVal2}R. Mikulevi\v{c}ius, G. Valiukevi\v{c}ius, On stochastic
Euler equation in $\mathbb{R}^{d}$, \textit{Electron. J. Probab.} \textbf{5}
(2000), no. 6, 20 pp.

\bibitem {Roz}B. L. Rozovskii, \textit{Stochastic Evolution Equations. Linear
Theory and Applications to Non-linear Filtering}. Kluwer, Dordrecht, 1990.

\bibitem {Sus}H. J. Sussmann, On the gap between deterministic and stochastic ordinary differential equations, \textit{Ann. Probab.} \textbf{6} (1978), no. 1, 19--41.

\bibitem {Temam_1983} R. Temam, Navier-Stokes equations and nonlinear functional analysis,
CBMS-NSF Regional Conference Series in Applied Mathematics, \textbf{41},  Society for Industrial and Applied Mathematics (SIAM), Philadelphia,
 PA,  1983.

\bibitem{Wolibner} W. Wolibner, 
Un theor{\`e}me sur l'existence du mouvement plan d'un fluide parfait, homog{\`e}ne, incompressible, pendant un temps infiniment long,
\textit{Math. Z.}  \textbf{37} (1933), no. 1, 698-726. 


\bibitem{Yokoyama} S. Yokoyama, Construction of weak solutions of a certain stochastic Navier-Stokes equation, \textit{Stochastics} \textbf{86} (2014), no. 4, 573--593.

\bibitem {Yud1}V. I. Yudovich, Non-stationary flows of an ideal incompressible fluid (Russian), \textit{\u{Z}. Vy\u{c}isl. Mat. i Mat. Fiz.} \textbf{3} (1963), 1032--1066.

\bibitem {Yud}V. I. Yudovich, Uniqueness theorem for the basic nonstationary
problem in the dynamics of an ideal incompressible fluid, \textit{Math. Res.
Lett.} \textbf{2} (1995), no. 1, 27--38.

\end{thebibliography}
\end{document}